\documentclass[12pt,reqno]{amsart}
\usepackage{hyperref,cite}
\usepackage{amssymb,amsthm,bbm}
\usepackage{amsfonts,cmtiup,comment,stmaryrd}
\usepackage{graphicx}
\usepackage{mathrsfs,cmtiup}
\usepackage{xcolor}

\makeatletter
\def\blfootnote{\xdef\@thefnmark{}\@footnotetext}
\makeatother

\newtheorem{theorem}{Theorem}[section]
\newtheorem*{theorema}{Theorem A}
\newtheorem*{theoremb}{Theorem B}
\newtheorem*{theoremc}{Theorem E}
\newtheorem*{theoremd}{Theorem F}
\newtheorem{lemma}[theorem]{Lemma}
\newtheorem{proposition}[theorem]{Proposition}
\newtheorem{corollary}[theorem]{Corollary}
\newtheorem*{corollaryc}{Corollary C}
\newtheorem*{corollaryd}{Corollary D}

\theoremstyle{definition}
\newtheorem{example}[theorem]{Example}

\newtheorem{remark}[theorem]{Remark}
\newtheorem*{definition*}{Definition}
\newcommand{\ed}{\end{document}}
\renewcommand{\geq}{\geqslant}
\renewcommand{\leq}{\leqslant}
\renewcommand{\ge}{\geqslant}
\renewcommand{\le}{\leqslant}
\let\le=\leqslant
\let\ge=\geqslant
\let\leq=\leqslant
\let\geq=\geqslant
\numberwithin{equation}{section}

\begin{document}

\title{Rank properties of commutators in finite groups}

\author{Cristina Acciarri}
\address{C.~Acciarri: Dipartimento di Scienze Fisiche, Informatiche e Matematiche, Universit\`a degli Studi di Modena e Reggio Emilia, Via Campi 213/b, I-41125 Modena, Italy}
\email{cristina.acciarri@unimore.it}

\author{Robert M. Guralnick}
\address{Robert M. Guralnick: Department of Mathematics, University of
Southern California, Los Angeles, CA90089-2532, USA}
\email{guralnic@usc.edu}

\author{Evgeny Khukhro}
\address{E. I. Khukhro: Charlotte Scott Research Centre for Algebra, University of Lincoln, U.K.}
\email{khukhro@yahoo.co.uk}

\author{Pavel Shumyatsky}

\address{P. Shumyatsky: Department of Mathematics, University of Brasilia, DF~70910-900, Brazil}
\email{pavel@unb.br}

\thanks{The first author is member of ``National Group for Algebraic and Geometric Structures, and Their Applications'' (GNSAGA–INdAM). The second author was partially supported by  Simons Foundation Fellowship 00019819.  The third author was partially supported by the International Center for Mathematics at SUSTech in Shenzhen. 
The fourth author was partially supported by  FAPDF and CNPq.}

\begin{abstract}
For a subset $S$ of a finite group $G$, let $I_G(S)$ denote the set of commutators $[g,x]=g^{-1}g^x$, where $g\in G$ and $x\in S$. Suppose that a finite group $G$ has a Carter subgroup $C$, that is, a nilpotent subgroup containing its normalizer.   Suppose that any subgroup generated by a subset of $I_G(C)$ is $r$-generated. We prove that if $G$ is soluble, then the derived subgroup $G'$ has $r$-bounded rank. We produce examples showing that the solubility condition cannot be dropped. For any finite group, we prove that the rank of $G'$ is $(r,l)$-bounded, where $l$ is the maximum rank of composition factors of $G$ isomorphic to $PSL_2(q)$ for $q\equiv 7\,(\operatorname{mod}8)$. We also prove in the general case that  $G'$ has $r$-bounded rank under the additional condition that  for any $x\in I_G(C)$, any subgroup generated by a subset of $I_G(x)$ is $r$-generated. The proofs rely on the classification of finite simple groups, using which we prove that if a finite simple group $G$ has an element $x$ of prime order $p$ such that  any subgroup generated by a subset of $I_G(x)$ is $r$-generated, then $G$ has $r$-bounded (Pr\"ufer) rank or is isomorphic to $PSL_2(q)$ with $q \equiv 3\,(\operatorname{mod}4)$ when $p=2$, or to $PSL_2(q)$  with $q \equiv -1\,(\operatorname{mod} p)$ when $p\ne 2$.
\end{abstract}

\maketitle

\section{Introduction}

By the rank of a finite group $G$ we mean the least positive integer $r$ such that every subgroup of $G$ can be generated by $r$ elements; this parameter is also called the Pr\"ufer rank. We emphasize that henceforth ``the rank'' always means ``the Pr\"ufer rank'', not to be confused with the Lie rank of a simple group of Lie type, for which we always use the ``Lie'' qualifier.
It is well known that bounds for the rank imply strong structural results for finite groups, as well as  for infinite groups in some classes. Conditions on the ranks are instrumental in the theory of powerful $p$-groups of A.~Lubotzky and A.~Mann  and in applications of this theory for profinite and residually finite groups. The ranks also appear as the dimensions of linear spaces arising as elementary abelian sections related to application of representation theory, like  Zassenhaus' theorem on soluble linear groups. Therefore bounds for the ranks are coveted results
in group theory.

In this paper we obtain a `global' bound for the rank of the derived subgroup $G'$ of a finite group $G$ from a `local' condition on the number of generators of subgroups generated by subsets of commutators of the form  $[g,x]=g^{-1}g^x$, where $g$ is an element of $G$ and $x$ is an element of a Carter subgroup $C$ of $G$. Naturally, such results make sense only if $G$ has Carter subgroups; this is always true when $G$ is soluble by R.\,W.~Carter's theorem~\cite{car}. E.\,P.~Vdovin \cite{vdo1, vdo2,vdo3,vdo4} extended Carter's results \cite{car} to arbitrary finite groups that have Carter subgroups: Carter subgroups are conjugate in any finite group in which they exist, and the image of a Carter subgroup in a quotient group is a Carter subgroup of the image. Vdovin \cite{vdo2} also obtained a complete description of almost simple groups that have Carter subgroups. (Henceforth a finite group $G$ is said to be almost simple if $S\leq G\leq \operatorname{Aut}S$ for a non-abelian simple group~$S$.)

We use the following notation to state our main results: for a subset $H$ of a group~$G$, we denote by $I_G(H)$ the set of commutators $[g,h]=g^{-1}g^h$, where $g\in G$ and $h\in H$, so that $[G,H]$ is the subgroup generated by $I_G(H)$.
Henceforth we write, say,  ``$(a,b,\dots)$-bounded'' to abbreviate ``bounded above by some function depending only on the parameters $a,b,\dots$".
\begin{theorema}
Let $r$ be a positive integer, $G$ a soluble finite group, and $C$ a Carter subgroup of $G$. Suppose that any subgroup generated by a subset of $I_G(C)$ can be generated by $r$ elements. Then the derived subgroup $G'$ has $r$-bounded rank.
\end{theorema}

Theorem~A does not hold without the assumption of solubility, as shown by examples produced in  \S\,\ref{counterexamples}.  However,  such examples exist only for a small class of almost simple groups, and it is possible to obtain a bound for the rank of $G'$ in terms of the ranks of these obstructions.

\begin{theoremb}
Let $r$ and $l$ be positive integers. Suppose that  a finite group $G$ contains a Carter subgroup $C$ such that any subgroup generated by a subset of $I_G(C)$ can be generated by $r$ elements. Let $l$ be the maximum
rank of composition factors of $G$ isomorphic to $PSL_2(q)$ for $q\equiv 7\,(\operatorname{mod}8)$. Then the derived subgroup  $G'$ has $(r,l)$-bounded rank.
\end{theoremb}

Note that the rank of $PSL_2(q)$, where $q=p^k$ for a prime $p$, is equal to $k$, unless $k=1$ and then the rank is~$2$.

Thus, in Theorems~A and B the `local' property of $r$-generation of any subgroup generated by a subset of $I_G(C)$ implies the `global' property of $r$-bounded (respectively $(r,l)$-bounded) generation of every subgroup of $G'$.

We state two corollaries of Theorem~B, which do not require restrictions on composition factors under additional hypotheses.

\begin{corollaryc}\label{c-odd}
Let $r$ be a positive integer. Suppose that a finite group $G$ has a Carter subgroup $C$ of order not divisible by~$8$ such that any subgroup generated by a subset of $I_G(C)$ can be generated by $r$ elements. Then the derived subgroup  $G'$ has $r$-bounded rank.
\end{corollaryc}

Another corollary also gives a global rank-type conclusion under a stronger local  condition; it is rather a consequence of the proof of Theorem~B.

\begin{corollaryd}\label{c-third}
Let $r$ be a positive integer. Suppose that a finite group $G$ has a Carter subgroup $C$ such that
\begin{enumerate}
\item[\rm (a)]  any subgroup generated by a subset of $I_G(C)$ can be generated by $r$ elements; and
\item[\rm (b)] for any $x\in I_G(C)$, any subgroup generated by a subset of $I_G(x)$ can be generated by $r$ elements.
\end{enumerate}
Then the derived subgroup  $G'$ has $r$-bounded rank.
\end{corollaryd}

The proofs of Theorem~B and Corollaries~C  and~D depend on the classification of finite simple groups. We use  this  classification for proving the following two results on finite simple groups, which are crucial in the proof of Theorem~B. These results  should be of general interest, and they are proved in greater generality than needed for application in the proof of Theorem~B.

\begin{theoremc}
 Let $r$ be a positive integer,  $G$ be a finite  simple group, and $x \in G$ an involution. Suppose that every subgroup generated by a subset of $I_G(x)$
can be generated by at most $r$ elements.  Then one of the following holds:
\begin{enumerate}
\item the  rank of $G$ is $r$-bounded; or
\item  $G = PSL_2(q)$ with $q \equiv 3\,(\operatorname{mod}4)$.
\end{enumerate}
\end{theoremc}

\begin{theoremd}
  Let $p$ be an odd prime, and $r$ a positive integer.  Let $G$ be a finite nonabelian simple group of order divisible by $p$, and $x\in G$ an element of order $p$.
  Suppose that every subgroup generated by a subset of $I_G(x)$
can be generated by at most $r$ elements.  Then one of the following holds:
\begin{enumerate}
\item the rank of $G$ is $r$-bounded; or
\item  $G=PSL_2(q)$  with
$q \equiv -1\,(\operatorname{mod} p)$.
\end{enumerate}
\end{theoremd}

Earlier we proved in \cite{agks}  `local--global' generation properties similar to Theorems~A and~B for the commutator set $I_G(P)$, where $P$ is a Sylow $p$-subgroup of $G$ (with a stronger result for $p$-soluble groups), and for $I_G(A)$, where $A$ is a group of coprime automorphisms of $G$. We  state these results, which are used in this paper, for convenience of the reader.

\begin{theorem}[{\cite[Theorem~1.2]{agks}}]\label{t-1-2}
Let $p$ be a prime, $r$ a positive integer, $G$ a  finite group, and $P$  a Sylow $p$-subgroup of $G$. Suppose that
\begin{enumerate}
\item[\rm (a)]  any subgroup generated by a subset of $I_G(P)$ can be generated by $r$ elements; and
\item[\rm (b)] for any $x\in I_G(P)$, any subgroup generated by a subset of $I_G(x)$ can be generated by $r$ elements.
\end{enumerate}
Then $[G,P]$ has $r$-bounded rank.
\end{theorem}

Recall that $\alpha$ is a coprime automorphism of a finite group $G$ if the order of $\alpha$ is coprime to the order of $G$, that is, $(|G|,|\alpha |)=1$. When $A$ is a group acting by automorphisms on a group $G$, the subset $I_G(A)$ and the subgroup $[G,A]$ have  the same meaning as above, in the natural semidirect product $G\rtimes A$.

\begin{theorem}[{\cite[Theorem~1.4]{agks}}] \label{t-1-4} Let $r$ be a positive integer. Suppose that $G$ is a finite group admitting a group of coprime automorphisms $A$ such that any subgroup generated by a subset of $I_G(A)$ can be generated by $r$ elements. Then $[G,A]$ has $r$-bounded rank.
\end{theorem}

Theorem~\ref{t-1-4} and the other results in \cite{agks}, as well as in the earlier paper by C.~Acciarri, R.~Guralnick, and P.~Shumyatsky \cite{tams}, can be regarded as a dual direction to the study of groups with small centralizers of elements or automorphisms. Indeed, the  set $I_G(H)$ is in a sense dual to the set of fixed points $C_G(H)$. For example, for a given element (or automorphism) $h$ we have $|I_G(h)|=|G:C_G(h)|$.

\begin{remark}\label{r-rem}
The result of Theorem~A for soluble groups was generalized in \cite[Corollary~C]{agks243} to the case of any $\pi$-subgroup in a $\pi$-soluble finite group. We included the proof of  Theorem~A in this paper for completeness, as this theorem is used in the proof of Theorem~B.
\end{remark}

\section{Preliminaries}

All groups considered in this paper are finite. By simple groups we always mean finite  non-abelian simple groups; we use their classification throughout.

\subsection*{Carter subgroups} First we recall the definition.

\begin{definition*}
  A \emph{Carter subgroup} $C$ of a finite group $G$ is a nilpotent subgroup that coincides with its normalizer: $C=N_G(C)$.
\end{definition*}

Clearly, if a group $G$ has a Carter subgroup $C$, then $C$ is a Carter subgroup of any subgroup of $G$ containing $C$. The following theorem was proved by R.\,W.~Carter \cite{car} for finite soluble groups and extended to arbitrary finite groups by E.\,P.~Vdovin \cite{vdo1, vdo2,vdo3,vdo4}. In what follows, we shall refer to Vdovin's results using his comprehensive paper  \cite{vdo4} including a corrigendum to the tables describing the Carter subgroups of almost simple groups.

\begin{theorem}[{\cite{car}, \cite{vdo1, vdo2,vdo3,vdo4}}]\label{t-cart}
Every soluble finite group has a Carter subgroup. If a finite group~$G$ has a Carter subgroup $C$, then all Carter subgroups of $G$ are conjugate and the image of $C$ in any quotient group $G/N$ is a Carter subgroup of $G/N$.
\end{theorem}

It is easy to deduce the following consequences.

\begin{corollary}\label{c-cart}
If a finite group $G$ has a Carter subgroup $C$, then $G$ is the normal closure of $C$ and $G'=[G,C]$. \end{corollary}

\begin{proof}
The image of $C$ in the quotient $G/\langle C^G\rangle$ is trivial and is a Carter subgroup of this quotient; so it is self-normalizing meaning that   $G/\langle C^G\rangle=1$. The image of $C$ in the quotient group $G/[G,C]$ is a Carter subgroup of it, which is central in $G/[G,C]$. This means that $G/[G,C]$ is abelian, and so $[G,C]=G'$.
\end{proof}

We also mention an elementary lemma showing that the `local' $r$-generation condition on $I_G(C)$ for a Carter subgroup $C$ is inherited by subgroups and quotient groups; we shall use these facts without special references.

\begin{lemma}\label{l-inher}
Let $r$ be a positive integer, and $G$ a finite group containing a Carter subgroup~$C$. Suppose that any subgroup generated by a subset of $I_G(C)$ can be generated by $r$ elements. Then
\begin{itemize}
  \item[\rm (a)] for any homomorphic image $\bar G$ of $G$ and any subset $C_1\subseteq C$, any subgroup generated by a subset of $I_{\bar G}(\bar C_1)$ can be generated by $r$ elements;
  \item[\rm (b)] for any subgroup $H\leq G$ and any subset $C_1\subseteq C$,  any subgroup generated by a subset of $I_{H}(C_1)$ can be generated by $r$ elements.
\end{itemize}
\end{lemma}

\subsection*{Automorphisms}
The following lemma about coprime actions  belongs to folklore (see, for example, \cite[I.18.6]{hup}); we shall sometimes use this fact without special references.

\begin{lemma}\label{l-1} Let $A$ be a group acting by automorphisms on a group $G$.  If $N$ is a normal $A$-invariant subgroup of $G$ of order coprime to $|A|$, then $C_{G/N}(A)=C_G(A)N/N$.
\end{lemma}

The next elementary lemma must also be well-known.

\begin{lemma}\label{l-3} Let $A$ be a nilpotent group acting by automorphisms on a nilpotent group $G$.  If $C_{G}(A)=1$, then $[G,A]=G$, and furthermore,  $G_p=[G_p,A_{p'}]$ for every prime $p$, where  $G_p$ is the Sylow $p$-subgroup of $G$, and $A_{p'}$ the Hall $p'$-subgroup of $A$.
\end{lemma}

\begin{proof}
  Let $A=A_p\times A_{p'}$, where $A_p$ is a Sylow $p$-subgroup of $A$. Then $C_{G_p}(A_{p'})=1$, because otherwise $A_p$ would have nontrivial fixed points on $C_{G_p}(A_{p'})$, which would belong to $C_{G}(A)$. We now have $G_p=[G_p,A_{p'}]$ by Lemma~\ref{l-1}.
\end{proof}

The following result of V.\,V.~Belyaev and B.~Hartley \cite{beha} is
a consequence of the classification of finite simple groups.

\begin{theorem}[{\cite[Theorem 0.11]{beha}}] \label{t-beha} If a finite group  $G$ admits a nilpotent group of automorphisms $A$ such that $C_G(A)=1$, then $G$ is soluble.
\end{theorem}

\subsection*{Rank} Recall that the (Pr\"ufer) rank of a finite group $G$ is the least positive integer $r$ such that every subgroup of $G$ can be generated by $r$ elements. The rank of an abelian, or more generally, nilpotent group is the maximum of the ranks of its Sylow subgroups.  A similar property is enjoyed by any finite group.

\begin{lemma}[{\cite{kov, gu2, lo-ma, luc}}]\label{l-rank-syl}
The rank of a finite group is at most $m+1$, where $m$ is the maximum of the ranks of its Sylow subgroups.
\end{lemma}

We also recall another well-known fact.

 \begin{lemma}\label{l-nil-rank}
 If $N$ is a nilpotent group of class $c$ generated by $k$ elements, then the rank of $N$ is bounded in terms of $k$ and $c$, and if $e$ is the exponent of $N$, then the order of $N$ is bounded in terms of  $k$,  $c$, and $e$.
 \end{lemma}

The following lemma appeared independently and simultaneously in
the papers of Yu.~M.~Gorchakov~\cite{grc}, Yu.~I.~Merzlyakov~\cite{me}, and as ``P.~Hall's lemma" in the paper of J.~Roseblade~\cite{rs}.

\begin{lemma}\label{l-gmh}
 Let $p$ be a~prime number. The rank of a~$p$-group of
automorphisms of an abelian
finite $p$-group of rank~$r$ is bounded in
terms of~$r$.
\end{lemma}

We also recall the following well-known fact (see, for example, \cite[Lemma~6.5]{agks}).

\begin{lemma}\label{l-rlt}
  If $G$ is a finite simple group of Lie type of Lie rank $d$ over the field $\mathbb{F}_{p^e}$ for a prime $p$, then the Pr\"ufer rank of $G$ is bounded in terms of $d$ and $e$.
\end{lemma}

Our objectives will usually be obtaining bounds for the ranks of finite simple groups involved; but once such a bound is already known, we shall be using their properties, which we collect in the following lemma.

  \begin{lemma}\label{l-113-3}
  Let $S$ be a finite simple group of rank $k$.

  \begin{itemize}
    \item[\rm (a)] If $S=A_m$ is an alternating group on $m$ symbols, then $m\leq 2k+3$.

    \item[\rm (b)] If $S$ is a group of Lie type of Lie rank $d$ over a field $\mathbb{F}_{p^f}$ for a prime $p$, then both $d$ and $f$ are $k$-bounded.

       \item[\rm (c)]  The order of the  outer automorphism group $ \operatorname{Aut}S/S$ is $k$-bounded.
  \end{itemize}
    \end{lemma}

  \begin{proof}
 The lemma is obviously true if $S$ is a sporadic group.

If $S=A_m$ is an alternating group, then $ |\operatorname{Aut}S/S|\leq 4$. A Sylow $2$-subgroup of the symmetric group $S_m$ contains $[m/2]$ commuting transpositions with disjoint support and therefore an elementary abelian $2$-subgroup of rank $[m/2]$. Hence $A_m$ contains an elementary abelian $2$-subgroup of rank $[m/2]-1$, and therefore $m\leq 2k+3$.

Now suppose that $S$ is a group of Lie type of Lie rank $d$ over a field $\mathbb{F}_{p^f}$. As shown in the proof of \cite[Lemma~6.4]{agks},  for $d>8$ the group $S$ contains an alternating group of degree at least~$d$ (this may also be a folklore result). Then  $d\leq 2k+3$ by part (a).  Since the Sylow $p$-subgroup of $S$ contains an elementary  abelian subgroup of order at least $p^f$, we obtain $f\leq k$.  It is well known that the order of the outer automorphism group  $ \operatorname{Aut}S/S$  is bounded in terms of $d$ and $f$. Since both $d$ and $f$ are $k$-bounded, therefore $|\operatorname{Aut}S/S|$ is also $k$-bounded.
  \end{proof}

Another useful property of simple groups of given rank is essentially a result, based on the classification of finite simple groups, of  J.~Hall, M.~Liebeck, and G.~Seitz \cite{hls} (later improved by R.~Guralnick and J.~Saxl \cite{GSaxl}) on generation of (almost) simple groups by conjugates.

 \begin{theorem}[{see \cite{hls} and \cite{GSaxl}}]
   \label{t-hls}
   If $S$ is a simple group of rank $k$, then for any nontrivial element $x\in S$ there are $k$-boundedly many conjugates of $x$ that generate $S$.
 \end{theorem}

\begin{proof}
  We only need to consider alternating groups and groups of Lie type. If $S$ is an alternating group, then it has $k$-bounded order by Lemma~\ref{l-113-3}(a). If $S$ is a group of Lie type, then its Lie rank $d$ is $k$-bounded by Lemma~\ref{l-113-3}(b).  The result of J.~Hall, M.~Liebeck, and G.~Seitz \cite{hls} (later improved by R.~Guralnick and J.~Saxl \cite{GSaxl}) states that there are $d$-boundedly many conjugates of $x$ that generate $S$.
\end{proof}

\section{The soluble case}
\label{soluble_section}

The aim of this section is to prove Theorem~A, which deals with  soluble groups. First we make use of our Theorem~\ref{t-1-4} about coprime actions.
For a prime $p$, the largest normal $p$-subgroup of a group $G$ is denoted by $O_p(G)$, and the largest normal $p'$-subgroup by $O_{p'}(G)$.

\begin{lemma}\label{l-2}
Let $p$ be a prime, and $K$ a $p'$-group acting on a group $G$ in such a manner that $[G,K]\leq O_p(G)$. Suppose that any subgroup generated by a subset of $I_G(K)$ is $r$-generator. Then $[G,K]$ has $r$-bounded rank.
\end{lemma}

\begin{proof}
 By Lemma~\ref{l-1} we have $G=C_G(K)O_p(G)$. Then $[G,K]=[O_p(G), K]$ and the result follows by Theorem~\ref{t-1-4}.
\end{proof}

Due to  Lemma~\ref{l-2}, the proof of the next lemma is essentially about finite $p$-groups. We will require the concept of powerful $p$-groups introduced by A.~Lubotzky and A.~Mann~\cite{LM}. A finite $p$-group $H$ is \emph{powerful} if and only if $[H,H] \leq H^p$ for $p\neq 2$ (or $[H,H]\leq H^4$ for $p = 2$). Apart from the original paper \cite{LM}, information about the properties of powerful $p$-groups can also be found in the book \cite{ddms}.

Recall that $F(G)$ denotes the Fitting subgroup of a group $G$, and  $\gamma_i(G)$ denotes the $i$-th term of the  lower central series of~$G$.

\begin{lemma}\label{last}
Let $M$ be a nilpotent subgroup of a group $G$ such that $[G,M]\leq F(G)$. Assume that any subgroup generated by a subset of $I_G(M)$ is $r$-generator. Then $[G,M]$ has $r$-bounded rank.
\end{lemma}

\begin{proof}
By Lemma~\ref{l-rank-syl} it is sufficient to prove that the rank of each Sylow $p$-subgroup of $[G,M]$ is $r$-bounded. Passing to the quotient of $G$ by $O_{p'}([G,M])$, we can assume that $[G,M]$ is a $p$-group.

Let $K=O_{p'}(M)$. Note that $[G,K]$ is contained in $O_p(G)$ (because $[G,M]\leq O_p(G)$ by the assumption). Then Lemma~\ref{l-2} shows that $[G,K]$ has $r$-bounded rank. Passing to the quotient $G/[G,K]$ we can assume that $K\le Z(G)$ and $[G,M]=[G,O_p(M)]$. This reduces the lemma to the case where $M$ is a $p$-group.

Therefore we assume henceforth that $M$ is a $p$-group and set $P=[G,M]M$, which is also a $p$-group.  Since $[G,M]=\langle I_G(M)\rangle$ is $r$-generated by hypothesis, by  the Burnside Basis Theorem \cite[5.3.2]{rob2} the commutator subgroup $[G,M]$ is generated by at most $r$ elements from $I_G(M)$, say, $[a_1,m_1],\dots ,[a_r,m_r]$  for $a_i\in G$, $m_i\in M$. Since $[a_i,m_i]\in M^{a_i}M$, it follows that $P=\langle M, M^{a_1},\dots,   M^{a_r}\rangle$. Introducing for convenience $a_0=1$, we will write $P=\langle M^{a_i}\mid  i=0,\dots,r\rangle$.

 Since $P=\langle M^{a_i}\mid i=0,\dots,r\rangle$, it follows that for any subgroup $N$ normalized by $P$ we have  $[N,P]=\prod_{i=0}^{r}[N,M^{a_i}]$. Indeed, the product on the right is $M^{a_i}$-invariant for every $i$ and $M^{a_i}$ acts trivially on the quotient of $N$ by this product; hence  $[N,P]\leq \prod_{i=0}^{r}[N,M^{a_i}]$, while the reverse inclusion is obvious. For any $a_i$ we have $I_G(M^{a_i})=I_G(M)^{a_i}$ and so, by hypothesis, each of the subgroups $[N,M^{a_i}]=\langle I_N(M^{a_i})\rangle$ is $r$-generator.
 As a result, $[N,P]$ is generated by at most $r(r+1)$ elements.
 In particular,  the terms of the lower central series $\gamma_j(P)$ are all $r(r+1)$-generator for $j\geq 2$.

Let $V$
denote the intersection of the kernels of all homomorphisms of $P$ into $GL_{r(r+1)}(\mathbb{F}_p)$ (see \cite[Definition~2.10]{ddms}). Put $W= V$ if $p$ is odd, and $W=V^2$ if $p=2$.  If $p$ is odd, then  $P/W$ is nilpotent of class at most $r(r+1)-1$, since it is isomorphic to a subgroup of  the direct  product of several copies of  the lower unitriangular group $U_{r(r+1)}(\mathbb{F}_p)$.  It follows that $\gamma_{r(r+1)}(P)\leq W$.  Since $\gamma_{r(r+1)}(P)$ is  $r(r+1)$-generator as shown above, it follows that $\gamma_{r(r+1)}(P)$ is a powerful  $p$-group by \cite[Proposition 2.12]{ddms}, and  the rank of $\gamma_{r(r+1)}(P)$ is at most $r(r+1)$. Since $[P,P]$ is also $r(r+1)$-generator,  the quotient $[P,P]/\gamma_{r(r+1)}(P)$ has $r$-bounded rank by  Lemma~\ref{l-nil-rank}.
As a result, the rank of $[P,P]$ is  $r$-bounded.
We can now pass to the quotient $G/[P,P]$ and assume that $P$ is abelian. Since $[G,M]$ is generated by $r$ elements, the rank of this abelian subgroup (contained in $P$) is at most $r$. This completes the proof for $p\ne 2$.

If $p=2$, then $P/W$ has an elementary abelian $2$-subgroup $V/W$ such that $P/V$ is nilpotent of class at most $r(r+1)-1$. Recall that $\gamma_{r(r+1)}(P)$ can be generated by $r(r+1)$ elements. Hence $\gamma_{r(r+1)}(P)/(W\cap\gamma_{r(r+1)}(P))$ is elementary abelian of order at most $2^{r(r+1)}$. Therefore, $\gamma_{2r(r+1)}(P)\leq W$. Since $\gamma_{2r(r+1)}(P)$ can be generated by $r(r+1)$ elements, again by \cite[Proposition 2.12]{ddms} it follows that $\gamma_{2r(r+1)}(P)$ is a powerful $2$-group of rank at most $r(r+1)$. Using the same argument as above for $p\ne 2$, we deduce that $[P,P]$ has $r$-bounded rank. Passing to the quotient $G/[P,P]$,  we similarly conclude that  the rank of $[G,M]$ is $r$-bounded. The proof is complete.
\end{proof}

We are now ready to embark on the proof of Theorem~A, which we restate here for the reader's convenience.

\begin{theorem}\label{t-sol}
Let $r$ be a positive integer, $G$ a soluble finite group, and $C$ a Carter subgroup of $G$. Suppose that any subgroup generated by a subset of $I_G(C)$ can be generated by $r$ elements. Then the derived subgroup $G'$ has $r$-bounded rank.
\end{theorem}

\begin{proof}
 Recall that $G$ is a finite soluble group,  and $C$ is a Carter subgroup of $G$ such that any subgroup generated by a  subset of $I_G(C)$ is $r$-generator.  We need to prove that $G'=[G,C]$ has $r$-bounded rank.

 First we show that  $G$ has a minimal normal subgroup of rank at most~$r$.

\begin{lemma}\label{clam1_minimalnormal}
The group  $G$ has a minimal normal subgroup of rank at most~$r$.
\end{lemma}

\begin{proof}Let $M$ be a minimal normal subgroup of $G$, which is an elementary abelian $p$-group for some $p$.  The rank of $[M,C]$ is at most $r$, since it is an abelian subgroup generated by a subset of $I_G(C)$ and therefore is $r$-generator by hypothesis. First suppose that  $C_M(C)=1$. Then
$M=[M,C]$ by Lemma~\ref{l-3}, and thus $M$ is a minimal normal subgroup of rank at most~$r$.

We  denote $A=C_M(C)$ for brevity and now assume that $A\ne 1$.  Since $C$ is self-normalizing, we have $A\leq C$. Hence $[G,A]$ is generated by a subset of $I_G(C)$ and therefore has rank at most $r$ by hypothesis, being an abelian subgroup contained in $M$. If $[G,A]\ne 1 $, then in fact $[G,A]=M$ is a minimal normal subgroup of rank at most~$r$.

 If $[G,A]=1$, then $A$ is central in $G$ and  any cyclic subgroup of $A$ is normal in $G$. This completes the proof.
\end{proof}

Recall that the Fitting series of $G$ starts with the Fitting subgroup $F_1(G) = F (G)$, and by induction $F_{i+1}(G)$ is the inverse image of $F (G/F_i(G))$. If $G$ is soluble, then the least number $h$ such that $F_h(G) = G$ is called the Fitting height of $G$. We also recall the notation $G^{(k)}$ for the $k$-th term of the derived series of a group $G$, starting from $G^{(0)}=G$ and $G^{(1)}=G'$.

We now return to the proof of Theorem~\ref{t-sol}.

\begin{lemma}\label{l-fh}
The Fitting height of $G$ is $r$-bounded.
\end{lemma}
\begin{proof}
Let $M$ be a minimal normal subgroup of $G$ of rank at most $r$, which exists by
Lemma~\ref{clam1_minimalnormal}.
Observe that the quotient group $G/C_G(M)$ acts faithfully  as a linear group on $M$ regarded as a vector space of dimension at most $r$ over the prime field $\mathbb{F}_q$ for some prime~$q$. Hence the derived length of $G/C_G(M)$  is $r$-bounded  by Zassenhaus's theorem  stating that the derived length of any soluble subgroup of ${GL}_n({k})$ over any field $k$ is bounded in terms of $n$ only (see \cite[15.1.3]{rob2}). Thus there is an $r$-bounded number $d=d(r)$ such that  $G^{(d)}\leq C_G(M)$, so that $[G^{(d)}, M]=1$.

We can repeat the same argument  for the quotient group $G/M$  with the same number  $d=d(r)$. We thus  find  a minimal normal subgroup $M_2/M$ of rank at most $r$ such that $(G/M)^{(d)}$  is contained in $C_{G/M}(M_2/M)$, so that  $[G^{(d)},M_2]\leq M$.  In this way we inductively construct  a normal series of $G$
$$1=M_0\leq M=M_1\leq M_2\leq \cdots \leq M_s=G$$
with the property that $[G^{(d)}, M_i]\leq M_{i-1}$ for $i\geq 1$.
Putting $K_i=G^{(d)}\cap M_i$, we obtain a normal series of $G^{(d)}$  such that $[G^{(d)}, K_i]\leq K_{i-1}$ for any $i\geq 1$. This means that $G^{(d)}$ is nilpotent. Since the quotient group $G/G^{(d)}$ is soluble of $r$-bounded derived length, we obtain that, in particular,  $h(G) \leq d+1$, so that the Fitting height of $G$ is $r$-bounded.
\end{proof}

We now proceed with the proof of Theorem~\ref{t-sol} by induction on the Fitting height $h(G)$ of $G$. Recall that $[G,C]=G'$. If $h(G)=1$, then $G$ is nilpotent and so $G=C$, since $C$ is self-normalizing. Then the result follows from Lemma~\ref{last} (or even from \cite[Lemma~4.1]{agks} applied to every Sylow subgroup of $G=C$).

Assume  now that $h(G)>1$.  Set for brevity $F=F(G)$, so the Fitting height of $G/F$ is smaller than $h(G)$ and, by induction, the image of $G'$ in $G/F$ has $r$-bounded rank.

Let  $C_0= C\cap F$. We have $[G,C_0]\leq F$, and obviously  $I_G(C_{0}) \subseteq I_G(C)$, whence   $[G,C_0]$ has $r$-bounded rank by Lemma~\ref{last}. Therefore we can pass to the quotient $G/[G,C_0]$ and assume that $C_0$ is central in $G$. Since the centre $Z(G)$ is contained both in $F$ and in $C$, in the quotient $\bar{G}=G/Z(G)$ the intersection of the images $\bar{F}\cap \bar{C}$  is the image of $C_0$, which is trivial, since $C_0$ is central. Hence $C_{\bar{F}}(\bar{C})=1$,  since $\bar{C}$ is self-normalizing being a Carter subgroup of  the quotient.   We claim that $\bar{F}$ has $r$-bounded rank.

It is sufficient to show that every Sylow $p$-subgroup $P$ of  $\bar{F}$ has $r$-bounded rank.  Let  ${K}$ be the Hall $p'$-subgroup of $\bar{C}$. Then ${P}=[{P},{K}]$ by Lemma~\ref{l-3}, since $C_{{P}}(\bar{C})=1$.  Now the rank of $[{P},{K}]$ is $r$-bounded by Theorem~\ref{t-1-4}.  Thus, $\bar F=F/Z(G)$ has $r$-bounded rank.

Recalling that the image of $G'$ in $G/F$ has $r$-bounded rank, we obtain that  $G'Z(G)/Z(G)$ also has $r$-bounded rank.  By a theorem of Lubotzky and Mann \cite[Theorem~4.2.3]{LM} then $G''$ has $r$-bounded rank.  Thus we can pass to the quotient by $G''$ and assume that $G$ is metabelian. In particular, then  $G=CG'$ and $[G',C]$ is a normal abelian subgroup. Since  $[G',C]$ is generated by $r$ elements by hypothesis,  $[G',C]$ has rank at most $r$. Finally, in the quotient by $[G',C]$, the derived subgroup $G'$ becomes central in $G=G'C$,  so $G$ becomes nilpotent of class~2. By Lemma~\ref{last} (or Lemma~\ref{l-nil-rank}) we conclude  that $G'=[G,C]$ has $r$-bounded rank.
\end{proof}

\section{Commutators in finite simple groups}\label{s-semisimple}

In this section we prove Theorems~E and~F about commutators with a fixed element of prime order in a finite simple group, from which we then derive consequences about Carter subgroups. We restate Theorem~E about involutions. Recall that
``the rank'' always means ``the Pr\"ufer rank'', while the Lie rank of a simple group of Lie type is always used with the ``Lie'' qualifier.

\begin{theorem}  \label{t-com-inv}
 Let $r$ be a positive integer,  $G$ be a finite  simple group, and $x \in G$ an involution. Suppose that every subgroup generated by a subset of $I_G(x)$
can be generated by at most $r$ elements.  Then one of the following holds:
\begin{enumerate}
\item the rank of $G$ is $r$-bounded; or
\item  $G = PSL_2(q)$ with $q \equiv 3\,(\operatorname{mod}4)$.
\end{enumerate}
\end{theorem}

\begin{proof}  We can ignore the sporadic simple groups.

Suppose that $G=A_n$ is an alternating group on $n$ symbols.
Let $m=\lfloor n/4\rfloor $. We can assume without loss of generality that $x=(12)(34)\cdots (4k-1\;4k)$ is the product of $2k$ transpositions with support $1,2,\dots,4k$. The commutators with $3$-cycles
$$\big[x, \;(4l-3\;4l-2\;4l-1)\big]=(4l-3\;4l-1)(4l-2\;4l) $$
for $1\leq l\leq k$ are  involutions in $I_G(x)$. These $k$ involutions have disjoint supports and therefore commute. If $k<m$, the commutators
$$\big[x, \; (1\;4s-3) (2\;4s-2) (3\;4s-1)(4\;4s)\big]=(12)(34)(4s-3\;4s-2)(4s-1\;4s)$$
for $k<s\leq m$ are also  involutions in $I_G(x)$. These $m-k$ involutions  commute with each other and with the first $k$ involutions. Altogether, these involutions generate an elementary abelian $2$-group of rank $m$. Hence $m\leq r$ by hypothesis, so $n\leq 4r+3$ and even the order of $G=A_n$ is $r$-bounded.

So we may assume that $G$ is a finite simple group of Lie type of Lie rank $d$ over the field $\mathbb{F}_{p^e}$ for some prime $p$. Recall that the rank of $G$ is bounded in terms of $e$ and $d$ by Lemma~\ref{l-rlt}.

First consider the case $G = PSL_2(p^e)$ for some prime $p$.   Let $U$ be a Sylow $p$-subgroup of order $p^e$ and let $B=N_G(U)$. If $p=2$, then we may assume that $x \in U$ and then $[x,B]=U$ is a  subgroup of rank $e$ generated by elements of $I_B(x)$. Then $e\leq r$ by hypothesis, whence the rank of $G$ is $r$-bounded.
If $p^e \equiv 1\,(\operatorname{mod}4)$,  then we can assume that $x \in B$ and then $I_B(x) = U$ and the same argument applies. The case $p^e \equiv 3\,(\operatorname{mod}4)$ is covered by part (2) of the theorem.

We next consider the remaining classical groups, that is,  $PSL_n(p^e)$ for $n > 2$,  $PSU_n(p^e)$ for $n >2$,  $PSp_{2n}(p^e)$ for $n > 1$, and $P\Omega_n ^{\varepsilon}(p^e)$ for $n > 6$.
We need to show that $n$ and $e$ are bounded in terms of $r$.   It is more convenient to work in the corresponding linear group, which we denote again by $G$.  Let $V$ be the natural module for $G$. With abuse of notation we denote by the same letter $x$ some $2$-element that is a pre-image of $x$. In characteristic $p=2$, this pre-image is an involution and is unique.

First suppose that $p=2$. It is sufficient to produce a $2$-subgroup $Q$ containing $x$ such that $[N_G(Q), x]\Phi(Q)/\Phi(Q)$ has rank at least $Ken^2$ for some fixed constant $K$. Since $[N_G(Q), x]$ is generated by elements of $I_G(x)$, we would then have $Ken^2\leq r$ by hypothesis, whence both $n$ and $e$ would be $r$-bounded, and so would be the rank of $G$.

If $G=SL_n(2^e)$, we can take for $Q$  the unipotent radical of the stabilizer of a subspace of dimension $n/2$ or $(n-1)/2$.  If $G=SU_n(2^e)$ or $Sp_{2n}(2^e)$, we can take for $Q$ the unipotent radical of the stabilizer of a totally singular subspace of maximal dimension.   In the case of~$\Omega_n ^{\varepsilon}(2^e)$,   we take $Q$ to be the unipotent radical of the stabilizer of a totally singular subspace $W$ of maximal dimension contained in the fixed space of $ x$; the dimension of $W$ will be close to half the dimension of the underlying space $V$ due to the choice of the pre-image~$x$. See \cite{GR1} for the elementary proof that such a $Q$ exists. Thus, both $n$ and $e$ are $r$-bounded in the case of classical groups in characteristic $p=2$.

Staying with classical groups over a field $\mathbb{F}_{p^e}$, we now assume that $p \ne 2$.   Recall that $V$ is the natural module of $G$ with $\dim V = n > 2$.   Since $G \ne SL_2$,  any lift of an involution is not semisimple regular and so it commutes with some unipotent element. Thus,  $x$ is contained in some maximal parabolic subgroup. Hence $x$ acts nontrivially on its unipotent radical $Q$ and so also on $Q/\Phi(Q)$. Since $x$ preserves the  $\mathbb{F}_{p^e}$-structure, it follows that $[x,Q]\Phi(Q)/\Phi(Q)$ requires at least $e$ generators and so $e \le r$.

It remains to show that $n$ is also $r$-bounded. First suppose that $x$ lifts to an element of order a multiple of $4$.  Then (over a quadratic field extension), $x$ has two distinct eigenvalues of the same multiplicity.  In particular, $n=2m$ is even.   Excluding the orthogonal case, we see that $x$ preserves an orthogonal direct sum of 2-dimensional nondegenerate spaces, on each of which  $x$ acts as  a noncentral element. On each of these 2-dimensional spaces, there is a commutator $[x,y]$ of  order $2$
(by inspection, or by Glauberman's $Z^*$-theorem). Hence there is an elementary abelian $2$-group of rank $m$ generated by elements of $I_G(x)$, and so $m \le r$ as required. If $G$ is orthogonal, we can pass to a subspace of codimension at most $2$ and assume that the group has $+$ type and $n=4m$.   We apply the same argument using nondegenerate spaces of dimension $4$ and again conclude that $m \le r$.

Now suppose that all eigenvalues of $x$ are $\pm 1$.    Let $m$ be the dimension of the largest eigenspace of $x$ (so $m \ge n/2$).   By passing to a subspace of dimension at
most $2$, we may assume that $m$ is even and contains a totally singular subspace of dimension $m/2$. (no condition is required for $SL$).  Then $V = W \oplus M \oplus W'$ as an $x$-module
where $\dim W = \dim W= m/2$ are totally singular, while $x$ is a scalar on $W \oplus W'$, and $x$ is a different scalar on $M$.
Since $x$ is not a scalar, we have $M \ne 0$.  Note that $W \otimes M$ is an $x$-invariant
quotient of the unipotent radical of the stabilizer of $W$.  Then $[x, W \otimes M] = W \otimes M$ and the rank of $W \otimes M$ is at least the rank of $W$, which  grows with
$n$, and the result follows in this case.

Finally we consider the case where $G$ is an exceptional group over a field $\mathbb{F}_{p^e}$.   Since the Lie rank of $G$ is bounded, it suffices to show that $e$ is $r$-bounded. When $p=2$, let   $ U$ be  a  Sylow $2$-subgroup of $G$ containing $x$, and let  $N_G(U)=UT$ for a torus $T$. Then $U$ has a $T$-invariant filtration in which each quotient $Q$ is a vector space over $\mathbb{F}_{p^e}$ such that $Q=[Q,T]$ (for $e > 1$).
Since the image $\bar x$ of $x$ in one of these quotients  is nontrivial, it follows that the rank of $[T,\bar x]$ as an elementary abelian $2$-group is at least $e$, whence $e\leq r$ by hypothesis. As a result, the rank of $G$ is $r$-bounded  in this case.

It remains to consider the case where $G$ is an exceptional group over $\mathbb{F}_{p^e}$ with $p$ odd.
The conjugacy classes of involutions and their  centralizers are  given in \cite{Lu}, and any involution commutes with a nontrivial unipotent subgroup.
Hence $x$ is contained in some maximal parabolic subgroup and so normalizes but does not centralize the unipotent radical of this parabolic subgroup.
Then $e$ is $r$-bounded as above.
\end{proof}

We now consider the case where $x$ is a fixed element of odd prime order $p$. Although we only
need the cases $p \le 3$ for our application to Carter subgroups, the general Theorem~F is proved here for all~$p$, with a bound for the rank that does not depend  on $p$.  The proof is fairly straightforward in the case where the element $x$ is contained in a parabolic subgroup.   If $x$ is regular semisimple, we use character theory and the formula for products of conjugacy classes to show
that $I_G(x) \cap U$  is large for some maximal unipotent subgroup $U$.  We then show that the subgroup generated by this set needs many generators if the rank is unbounded. One can give  a different proof in the case where $x$ is regular using Steinberg's description of regular elements \cite{St}. We restate Theorem~F.

\begin{theorem}\label{t-com-odd}
  Let $p$ be an odd prime, and $r$ a positive integer.  Let $G$ be a finite simple group of order divisible by $p$, and $x\in G$ an element of order $p$.
  Suppose that every subgroup generated by a subset of $I_G(x)$
can be generated by at most $r$ elements.  Then one of the following holds:
\begin{enumerate}
\item the  rank of $G$ is $r$-bounded; or
\item  $G=PSL_2(q)$  with
$q \equiv -1\,(\operatorname{mod} p)$.
\end{enumerate}
\end{theorem}

\begin{proof}  We can ignore the sporadic simple groups.

First suppose that $G=A_n$ is an alternating group on $n$ symbols. Let $m=\lfloor n/p\rfloor $. We can assume without loss of generality that $x=(12\dots p)\cdots ((k-1)p+1\dots kp)$ is the product of $k$ cycles of length $p$  with support $1,2,\dots ,kp$ for some $k\leq m$. An argument similar to that for $p=2$ in the beginning of the proof of Theorem~\ref{t-com-inv} shows that there are elements in $I_G(x)$ that generate an elementary abelian $p$-group of rank $m$, so that $m\leq r$ by hypothesis and therefore $n\leq pr+p-1$. We claim that in addition $p$ is $r$-bounded.  The commutators $[x,\, (4k+1\; 4k+2\; 4k+3)]=(4k+1\;4k+2\; 4k+4)$
 for $k=0,1,\dots ,\lfloor p/4\rfloor-1$ belong to $I_G(x)$. Since these 3-cycles are independent, they generate an elementary abelian 3-group of rank $\lfloor p/4\rfloor$, whence $p$ is $r$-bounded by hypothesis, as claimed. As a result, $n$ is also $r$-bounded, so that even the order of $G=A_n$ is $r$-bounded.

So we may assume that $G$ is a finite simple group of Lie type of Lie rank $d$ over a  field~$\mathbb{F}_{s^e}$ in characteristic~$s$. Recall that the rank of $G$ is bounded in terms of $e$ and $d$ by Lemma~\ref{l-rlt}.
It is more convenient to work in the corresponding linear group, which we denote again by~$G$; let $V$ be the natural module of dimension $n$  for $G$. With abuse of notation we denote by the same letter $x$ some $p$-element that is a pre-image of $x$. In characteristic $p$, the lift still has order $p$. In other cases we choose $x$ to be a lift (of our element of order $p$) with maximal dimension $\dim V_0$ of the fixed-point subspace $V_0$ of $x$.
\medskip
\reversemarginpar

\paragraph*{\bf Case $\boldsymbol{p=s}$} Suppose  that $p=s$.  We first show that $e$ is $r$-bounded.   By \cite[Lemma 2.1]{GSaxl}, either the twisted Lie rank of $G$ is $1$, or $x$ is contained in a maximal parabolic subgroup $P$ and is not contained
in its unipotent radical $Q$.  Then we can pass to $P/Q$ and so by induction
assume that  the twisted Lie rank of $G$ is $1$.   Let $B=TU$ be the Borel subgroup containing $x$ with
$T$ a torus and $U$ the unipotent radical, so that $x\in U$.    We can choose $T$-invariant unipotent subgroups $U_1$, $U_2$ with $U_2$ a proper normal subgroup of $U_1$ such
that $x\in U_1\setminus U_2$.
Then $I_T(x)$ generates a nontrivial $T$-submodule of $U_1/U_2$, which
is a vector space
over $\mathbb{F}_{s^e}$.  Thus $e$ is $r$-bounded.

Staying in the case $p=s$, we next show that $d$ (or, equivalently, $n$) is $r$-bounded.  We may assume that $d >8$ and so assume that $G$ is a classical group.

We  decompose $V$ into an orthogonal
direct sum of $x$-invariant subspaces $V_i$, $0  \le i \le t$, such that the element $x$ acts trivially on $V_0$ (possibly, $V_0=0$) and  $x$ preserves no proper nondegenerate
subspace of $V_i$ for $i > 0$.  Since $x$ is unipotent, $x$ acts on each  $V_i$ either as a single Jordan block or as two Jordan blocks of the same size (see \cite[Chapter~3]{LS}), which may be different for different $V_i$
(for $G$ of type $SL$ we assume that all subspaces are  nondegenerate and $x$ has a single Jordan block on each $V_i$);
cf.   \cite[Chapters~3,~4]{LS}.
Let $m_i = \dim V_i$, and let $x_i$ be the projection of $x$ on $V_i$.    If $B_i$ is a Borel subgroup of the isometry group on $V_i$ (with $x_i \in B_i$), then
there is a subgroup generated by elements of  $I_{B_i}(x_i)$ that needs $m_i-1$ generators (if $x_i$ is a single Jordan block) or $m_i/2-1$ generators (if $x_i$ has  two Jordan blocks).  As a result, there is a $p$-subgroup generated by elements of $I_G(x)$ that requires at least $\big(\sum_{i=1}^t  m_i\big)/4$ generators, and this number must be at most $r$ by hypothesis.   Thus, the result follows unless, say,  $\sum_{i=1}^t m_i < n/2$ and then  $\dim V_0 > n/2$.
By passing to a subspace of codimension at most $2$, we may assume that $V_0$ is the direct sum of two complementary totally isotropic spaces $W \oplus W'$  (in the case of $SL$ we just take $W=V_0$).  Without loss of generality, we have  $\dim W \ge n/4$.  Let $P$ be the maximal parabolic subgroup stabilizing~$W$.   Then $x$ is in the Levi subgroup of $P$ and acts trivially on $W$ and nontrivially on $M:=\bigoplus_{i>0} V_i$.      The unipotent radical $Q$ of $P$ has an elementary abelian quotient    that
is isomorphic to $W\otimes M$ as an $\mathbb{F}_{s^e}\langle x\rangle$-module.     In particular, some elements of $I_Q(x)$ generate a subgroup requiring at least $n/4$ generators, so $n/4\leq r$, and the result follows.

\medskip
We now assume that $s \ne p$.  The arguments are different for the cases where  $x$ is contained in some parabolic subgroup, or not.

\medskip

\paragraph*{\bf Case  $\boldsymbol{p\ne s}$ and $\boldsymbol x$ is in a parabolic
subgroup} Suppose that $x$ is contained in a parabolic
subgroup $P$ with unipotent radical $Q$.   Since $F^*(P)=Q$, the element $x$ acts nontrivially on $Q$ and so on $Q/\Phi(Q)$ (since the order of $x$ is coprime to $s$).  Since we may view $Q/\Phi(Q)$ as an $\mathbb{F}_{s^e}\langle x\rangle$-module, we have $|[x,Q]| \ge s^e$ and so $e\leq r$.

Thus, in this case of $x$ in a  parabolic
subgroup, it remains to bound the Lie rank $d$, or equivalently, $n=\dim V$. We may assume that $d >8$ and so assume that $G$ is a classical group.

First assume that the fixed space $V_0$ of $x$ contains a nonzero isotropic vector (or any nonzero vector in the case of $SL$).   By passing to
a subspace of codimension at most $2$, we may assume that $V_0 = W \oplus W'$ with $W$ and $W'$ totally isotropic (if $G=SL$, take $W=V_0$).  Let $M$ be an $x$-invariant complement to $W$ in $W^{\perp}$.   Let $P$ be the parabolic subgroup stabilizing $W$, and  $Q$  the unipotent radical of $P$.  Then a quotient $\bar{Q}$ of $Q$ is isomorphic to $W\otimes M$ as  an $\mathbb{F}_{s^e}\langle x\rangle$-module. Since $[x,M]=M$, we see that there exists a subgroup generated by elements of $I_G(x)$ generating a group requiring $e(\dim W )(\dim M) $ generators.  Since $\dim W + \dim M \ge \dim V/2$,
the result follows.

If $V_0$ has no isotropic vectors, then $\dim V_0 \le 2$ and so by passing to an $x$-invariant subspace of $V$ of codimension at most $2$ we can assume that $V_0=0$. We claim that all nontrivial irreducible
 $x$-invariant subspaces of $V$ have
the same dimension.
Indeed, let $x^p=\alpha I$. If $\alpha\ne \beta ^p$ for any $\beta\in \mathbb{F}_{s^e}$, then the polynomial $X^p-\alpha$ is irreducible over $\mathbb{F}_{s^e}$ and then all $x$-irreducible subspaces have dimension $p$. If $\alpha= \beta ^p$ for  $\beta\in \mathbb{F}_{s^e}$, and $ \mathbb{F}_{s^e}$ contains a primitive $p$-th root of unity, then $x$ is diagonalizable (over the same field). If $\alpha= \beta ^p$ for  $\beta\in \mathbb{F}_{s^e}$, but  $ \mathbb{F}_{s^e}$ does not contain a primitive $p$-th root of unity, then $\beta$ is a unique $p$-th root of $\alpha $ in $\mathbb{F}_{s^e}$ and $X^p-\alpha=(X-\beta)f_1(X)\cdots f_k(X)$, where the $f_i(X)$ are  irreducible factors, all having the same degree $d$, over  $\mathbb{F}_{s^e}$; then all irreducible subspaces either have dimension 1 or dimension~$d$. Then $\beta ^{-1}x$ is another lift with all 1-dimensional irreducible subspaces in $V_0$. By the choice of $x$ with maximal $\dim V_0$, in this case we must have $x^p=1$ and all nontrivial irreducible subspaces of $x$ are of dimension $d$.

\normalmarginpar

Thus, all nontrivial irreducible
$x$-invariant subspaces of $V$ have
the same dimension. If this dimension is $1$,  then we argue as in the previous theorem for the case of an involution.
Thus, we may assume that any nontrivial $x$-invariant space has
dimension at least $2$.  Recall that we need to  prove
that
$d$ (or $n$) is bounded in terms of $r$.

Since $x$ is contained in  a parabolic subgroup,  $x$ preserves some totally isotropic subspace~$W$,  and we may assume that $x$ acts irreducibly (and nontrivially) on $W$.
Then as above we can write
$V=W \oplus M \oplus W'$, where  $M$ is an $x$-invariant complement to $W$ in
$W^{\perp}$ and $W'$ is another isotropic $x$-invariant subspace
(it is possible that $M=0$).
Let $P$ be the stabilizer of $W$,  and  $Q$ the unipotent radical of $P$.
\reversemarginpar

  If $M=0$, then   $W'$ is the dual of $W$ (or twisted in the case of $SU$)  as $\mathbb{F}_{s^e}\langle x\rangle$-modules, and $Q$  is  isomorphic as an $\mathbb{F}_{s^e}\langle x\rangle$-module to the
submodule of
the tensor product $W\otimes W'$, which is the whole of $W\otimes W'$ in case $G=SL$, or given by the symmetric product, the alternating product, or Hermitian matrices
(in the symplectic, orthogonal, or unitary cases, respectively).  We claim that $\dim C_{ Q}(x) \leq  \dim W$. Indeed, after extending
scalars to the algebraic closure, we have $W=\bigoplus U_i$ for $1$-dimensional
distinct irreducible  $\mathbb{F}_{s^e}\langle x\rangle$-modules $U_i$, while $W'=\bigoplus U_i^*$. Then
$$
\dim C_{ Q}(x)\leq \dim C_{W\otimes W'}(x)=\dim \bigoplus C_{U_i\otimes U_i^*}(x)=\dim W,
$$
since $\dim C_{U_i\otimes U_i^*}(x)=1$ and $C_{U_i\otimes U_j^*}(x)=0$ for all $i\ne j$.  Since $|x|$ is coprime  to $|Q|$, it follows that $\dim [ Q,x] \geq  \dim Q - \dim W$, and since $\dim  Q$ is roughly at least $(\dim W)^2=(\dim V/2)^2$, while  the rank of $[ Q,x]$ is at most $r$ by hypothesis, we obtain that $n=\dim V$ is bounded in terms of~$r$.

Suppose that $M\ne 0$.    Then $Q$ has an $x$-invariant quotient $\bar
Q$  isomorphic to
$W \otimes  M$ as an $\mathbb{F}_{s^e}\langle x\rangle$-module.  Note that $M=\bigoplus_{i=1}^t U_i$
is a direct sum of $t$ irreducible $\mathbb{F}_{s^e}\langle x\rangle$-modules $U_i$, each  of dimension equal to $\dim W$, for some $t\geq 1$.    Since $V= W \oplus M \oplus W'$, we have $n=\dim V = 2 \dim W + \dim M=(t+2)\dim W$.  In
particular, at least one of the numbers $t$ and $\dim W$ grows with $n$.
As above, decomposing $W$ and $U_i$ over algebraic closure of the field into sums of 1-dimensional irreducible submodules, which are all distinct for each of $W$ and~$U_i$, we obtain $\dim C_{W\otimes U_i}(x)  \leq \dim W$ for each $i$. Hence,
$$
\dim C_{\bar Q}(x)= \dim C_{W \otimes M}(x) \leq t \dim W.
$$  As a result,
$\dim [\bar Q,x]\geq   \dim \bar Q -t\dim W=t(\dim W)^2-t\dim W$,  and this value grows with~$n$. Since the rank of $[\bar Q,x]$ is at most $r$, we obtain that $n$ is $r$-bounded.

\medskip

\paragraph*{\bf Case  $\boldsymbol x$ is not in a parabolic
subgroup}

Finally assume that $x$ preserves no isotropic subspace (in particular $p \ne s$)  and thus $x$ is not contained in any parabolic subgroup. This implies that
$x$ is a regular semisimple element.

If $G = PSL_2(s^e)$, then $s^e \equiv - 1\, (\operatorname{mod} p)$, since $x$ is not contained in a Borel subgroup, and the result holds as  part (2) of the theorem .

We proceed with other groups of Lie type; let $q=s^e$ to lighten some formulae.

We first show that the Lie rank is $r$-bounded; recall that we are considering the case where  $x$ is not in any parabolic subgroup.  We may assume that $G$ is a classical
group.

Consider the set of triples $(a,b,c)$ \marginpar{\bf C4} of elements of $G$ with
$abc=1$ such that $a$ is  conjugate to $x^{-1}$, \ $b$ is conjugate to
$x$, and $c$ is conjugate to a fixed regular unipotent element $u$.

By \cite[Theorem~7.2.1]{serre},
the number $N$ of such triples
is equal to
\begin{equation}\label{e3}
  \frac{1}{|G|}\cdot |x^G|\cdot |x^G|\cdot |u^G|\cdot  \sum_{\chi}\frac{\chi (x^{-1})\chi (x)\chi(u)}{\chi(1)}
\end{equation}
with summation over all irreducible characters of $G$.  Let
$$
T = \frac{1}{|G|}\cdot |x^G|\cdot |x^G|\cdot |u^G|.
$$

 Fix any $\epsilon > 0$.
 By \cite[Theorem 1.3]{GLT},
 for sufficiently large rank,  $|\chi(g)|
\le \chi(1)^{1/4}$ for any regular element $g\in G$ and any   irreducible character~$\chi$.
Therefore from \eqref{e3} we obtain
$$N \ge  T \left(1-\sum_{\chi\ne \mathbbm{1}_G } \frac{|\chi (x^{-1})\chi (x)\chi(u)|}{\chi(1)}\right)\geq T \left(1-\sum_{\chi\ne \mathbbm{1}_G } \chi(1)^{-1/4}\right).$$
By  \cite[Theorem 1.2]{LSh} we have $\sum_{\chi\ne \mathbbm{1}_G } \chi(1)^{-1/4}\to 0$ as $|G|\to \infty$ for large enough rank of $G$.
In fact, we need a more precise assertion that $\sum_{\chi\ne \mathbbm{1}_G } \chi(1)^{-1/4}\to 0$ as $d\to \infty$. Although not explicitly mentioned in the statement of \cite[Theorem 1.2]{LSh}, this fact is actually proved in \cite[Theorem 1.2]{LSh}, where it is shown that for any $t>0$, for large enough Lie rank, we have $$\sum_{\chi\ne \mathbbm{1}_G } \chi(1)^{-t}\leq c_4q^{2/t}c_2^{\sqrt{n}}\cdot q^{-t(n-1)/2} + c_5D^{-t}q^{-n}$$
for some positive constants $c_2,c_4,c_5, D$, where $n$ is the dimension of the natural module for~$G$ (which is $d+1$, $2d$, or $2d+1$). The right-hand side expression clearly tends to $0$ as $n\to \infty$, that is, as $d\to \infty$. (See the proof of \cite[Theorem 1.2]{LSh}.)
Therefore
 $N \ge  T  (1-\epsilon)$ for large enough rank of $G$.

We now count the solutions with $a = x^{ -1}$ fixed; note that for such triples we have  $c=u^v=(x^{-1}x^w)^{-1}=[w,x]\in I_G(x)$ for some $v,w\in G$, and the number of these triples $(x^{-1},x^w,u^v)$ with $x^{-1}x^wu^v=1$ is the same as the number of their elements  $u^v$, which are $s$-elements in $I_G(x)$. We see that  the number of such solutions is
$$N/|x^G|\geq (1 - \epsilon) |x^G||u^G|/|G|.$$
 We use this inequality  to get a lower bound for $|I_G(x) \cap U|$ for some Sylow
$s$-subgroup $U$.  The number of Sylow $s$-subgroups is $|G:B|$ where
 $B=N_G(U)$ is a Borel subgroup. Hence for at least one Sylow $s$-subgroup $U$ we must have  $$|I_G(x) \cap U| \ge \frac{(1 - \epsilon) |x^G||u^G|/|G|}{|G|/|B|}\geq  (1-
\epsilon) \frac{|B|}{|C_G(x)||C_G(u)|}.$$
The number on the right-hand side  is at least
 (roughly) $|U|/q^{2d}$, because $|B|$ is roughly equal to $|U|\cdot |C_G(x)|$ (in fact $|B|=|U|(q-1)^d$, since $x$ is regular semisimple
 while $|C_G(x)|\leq (q+1)^d$), and $|C_G(u)|=q^{d}|Z(G)|$
 since   $u$ is regular unipotent.

 Suppose that $A$ is an $r$-generated subgroup of $U$.   We claim that
$|U:A| > q^{2d}$ for large enough $d\gg r$.
Consider $U/\gamma_4(U)$,
where $\gamma_4(U)$ is the fourth term in the lower central
series of~$U$.  This is an $s$-group of nilpotency
class  $3$
and of exponent $s$ (or possibly $s^2$ for very small~$s$).   Thus
$A\gamma_4(U)/\gamma_4(U)$ has order at most $s^{f(r)}$.  Since $U/\gamma_4(U)$
has order very close to $q^{3d}$,
the claim follows.

For large enough $d\gg r$ we obtain a contradiction: on the one hand, $|I_G(x) \cap U|\geq
 |U|/q^{2d}$, whence the index of the subgroup $\langle I_G(x) \cap U\rangle$ is at most  $q^{2d}$, while on the other hand its index is greater than  $q^{2d}$, since this subgroup is $r$-generated by hypothesis.

Thus, the Lie rank $d$ is $r$-bounded. It remains to show that $e$ is
$r$-bounded, where, recall, our group $G$ is a finite simple group of Lie type of Lie rank $d$ over the field $\mathbb{F}_{q}$ of characteristic $s$ with $q=s^e$.  Since we know that the rank is bounded, it suffices to assume that
the rank is fixed and indeed
the type of $G$ is fixed (with $e$ increasing).  We exclude the case
$SL_2(q)$ as this is already covered.

Let $\chi$ be any irreducible character of $G$. Given the fixed Lie type and Lie rank,
we shall be using the fact that
\begin{equation}\label{e1}
|\chi(x)|\leq C_1
\end{equation}
for an absolute constant $C_1$
independent of $q$,  since $x$ is regular semisiple  (see \cite[Theorem 5.4]{GM} or \cite{ltt}). Furthermore,
\begin{equation}\label{e2}
 \sum _i\chi(u_i)\leq C_2
\end{equation}
for an absolute constant $C_2$
independent of $q$, where the $u_i$ are representatives of
distinct conjugacy classes of regular unipotent elements (by \cite{LGL} in good characteristic, and by \cite[2.14]{Ge} in general).

For any representative $u_i$ of a
conjugacy class of regular unipotent elements, consider the
triples $(a,b,c)$ such that $abc=1$ with $a\in (x^{-1})^G$, $b\in x^G$, and $c\in u_i^G$.  By \cite[Theorem~7.2.1]{serre}, the number of such triples is equal to
\begin{equation*}
 \frac{1}{|G|}\cdot |x^G|\cdot |x^G|\cdot |u_i^G|\cdot  \sum_{\chi}\frac{\chi (x^{-1})\chi (x)\chi(u_i)}{\chi(1)}
\end{equation*}
with summation over all irreducible characters of $G$.
Then the number of such triples with $a=x^{-1}$ fixed is
\begin{equation*}
N_i= \frac{1}{|G|}\cdot  |x^G|\cdot |u_i^G|\cdot  \sum_{\chi}\frac{\chi (x^{-1})\chi (x)\chi(u_i)}{\chi(1)}.
\end{equation*}
Note that for such triples, we have  $c=u_i^v=(x^{-1}x^w)^{-1}=[w,x]\in I_G(x)$ for some $v,w\in G$, and the number of these triples $(x^{-1},x^w,u_i^v)$ with $x^{-1}x^wu_i^v=1$ is the same as the number of their elements  $u_i^v$, which are $s$-elements in $I_G(x)$. Let $k$ be the number of distinct conjugacy classes of regular unipotent elements of $G$. We can choose one of regular unipotent elements $u$ and write  $|u^G|=|u_i^G|=|G|/|C_G(u_i)|=|G|/q^d|Z(G)|$ for all $i$. We obtain
\begin{align*}
N:=\sum_{i=1}^kN_i&= \sum_{i=1}^k \Big(\frac{ |x^G|\cdot |u^G|}{|G|}\cdot    \sum_{\chi}\frac{\chi (x^{-1})\chi (x)\chi(u_i)}{\chi(1)}\Big)\\ &=
\frac{ |x^G|\cdot |u^G|}{|G|}\cdot    \sum_{\chi}\frac{\chi (x^{-1})\chi (x)\sum_{i=1}^k\chi(u_i)}{\chi(1)}.
\end{align*}
Note that $N$ counts (some of the) different $s$-elements in $I_G(x)$.
Using the  inequalities \eqref{e1} and \eqref{e2}, we obtain
$$
N\geq \frac{ |x^G|\cdot |u^G|}{|G|}\cdot \Big(k-C_1^2C_2\sum_{\chi\ne \mathbbm{1}_G}\chi(1)^{-1}\Big).
$$

Since the Coxeter number $h$ of our group $G$ is greater than $2$ by the assumption that $G$ is not $SL_2(q)$, so that $1>2/h$, by  \cite[Theorem 1.1]{LSh} we have $\sum_{\chi\ne \mathbbm{1}_G } \chi(1)^{-1}\to 0$ as $|G|\to \infty$. Therefore, for any positive $\epsilon$, we have
$$
N\geq \frac{ |x^G|\cdot |u^G|}{|G|}\cdot (k-\epsilon)
$$
if $e$ is large enough.

We use this inequality  to get a lower bound for $|I_G(x) \cap U|$ for some Sylow
$s$-subgroup $U$.  The number of Sylow $s$-subgroups is $|G:B|$ where
 $B=N_G(U)$ is a Borel subgroup. Hence for at least one Sylow $s$-subgroup $U$ we must have
 $$|I_G(x) \cap U| \ge \frac{(k - \epsilon) |x^G||u^G|/|G|}{|G|/|B|}\geq  (k-
\epsilon) \frac{|B|}{|C_G(x)||C_G(u)|}\geq  (1-
\epsilon) \frac{|B|}{|C_G(x)||C_G(u)|}.$$
The number on the right-hand side  is at least
 (roughly) $|U|/(d+1)q^{d}$, because $|B|$ is roughly equal to $|U|\cdot |C_G(x)|$  (in fact $|B|=|U|(q-1)^d$ since $x$ is regular semisimple,
 while  $|C_G(x)|\leq (q+1)^d$),  and $|C_G(u)|=q^{d}|Z(G)|\leq q^{d} (d+1)$
 since   $u$ is  regular unipotent.

By hypothesis, $\langle I_G(x) \cap U\rangle$ is an $r$-generated subgroup of $U$.   Since $U$ is a nilpotent  $s$-group of class at most $d-1$ and of exponent bounded in terms of $s$ and $d$, the order of $\langle I_G(x) \cap U\rangle$ is bounded above in terms of $r,s,d$ by Lemma~\ref{l-nil-rank}. On the other hand, since we excluded the case $G=SL_2$, we have $|U|\geq q^{d+1}$. This leads to a contradiction for large enough $q\gg d$ in view of the inequality $|I_G(x) \cap U| \ge |U|/(d+1)q^d$ proved above. Hence $e$ is bounded in terms of $d$, and therefore also in terms of $r$, as required.
\end{proof}

\medskip

We now derive some consequences of Theorems~\ref{t-com-inv} and \ref{t-com-odd} for Carter subgroups, which will be used in the next section in the proofs of Theorem~B and Corollaries~C and~D. Recall that the generalized Fitting subgroup $F^*(G)$ of a group $G$ is  equal to $F(G)E(G)$, where $F(G)$ is the Fitting subgroup, while $E(G)=Q_1*\dots *Q_k$, known as the layer, is the central product of all subnormal quasisimple subgroups $Q_i$, that is, perfect groups with non-abelian simple central quotients.  Recall also that a finite group $G$ is said to be almost simple if $S\leq G\leq \operatorname{Aut}S$ for a non-abelian simple group $S$ identified with the subgroup of inner automorphisms; then of course  $F^*(G)=S$.

\begin{corollary}\label{c-simple}
  Let $G$ be an almost simple group with a Carter subgroup $C$.   Assume that every subgroup of $G$ generated by
a subset of $I_G(C)$ is $r$-generated.  Then either the  rank of~$G$ is $r$-bounded or $G = PSL_2(q)$ with $q \equiv 7\,(\operatorname{mod}8)$.
\end{corollary}

\begin{proof}    We note that by \cite{vdo4}, either $C \cap F^*(G)$ has even order or  $F^*(G)=PSL_2(3^{2k+1})$. When  $C \cap F^*(G)$ has even order,  there exists an involution $x \in C \cap F^*(G)$. Then the result follows by Theorem~\ref{t-com-inv} unless possibly $F^*(G) = PSL_2(p^e)$ with $e$ odd and $p$ a prime
with $p \equiv  3\,(\operatorname{mod}4)$.  If $p \equiv 3\,(\operatorname{mod}8)$, then again by \cite{vdo4},   $G$ must contain $PGL_2(p^e)$ and $C$ contains
a Sylow $2$-subgroup.   Then there is an involution $x \in C$
inducing inversion on the Sylow $p$-subgroup $U$ and so
$I_U(x)=U$, whence $e \le r$.    Thus, the only possibility is that $p \equiv 7\,(\operatorname{mod}8)$ and the result follows.

It remains to consider the case where $F^*(G)=PSL_2(3^{2k+1})$. By \cite{vdo4} then $C$ contains a Sylow $3$-subgroup of $C_{PSL_2(3^{2k+1})}(\zeta_{3'})$, where $\zeta_{3'}$ is the $3'$-part of a field automorphism $\zeta$ of odd order.
We choose an element $x\in C$ of order~3, for which any subset of  $I_G(x)$ generates an $r$-generator subgroup by hypothesis.  Then the rank of $G$ is $r$-bounded by Theorem~\ref{t-com-odd}, since  $3^{2k+1}\not\equiv -1\,(\operatorname{mod} 3)$
\end{proof}

We now derive from Corollary~\ref{c-simple} two technical lemmas about Carter subgroups in groups of special form.

\begin{lemma}\label{112} Let $G$ be a group containing a Carter subgroup $C$ such that any subgroup generated by a subset of $I_G(C)$ is $r$-generator. Suppose that $G=MC$, where $M$ is a direct product of simple groups, each normal in $G$.  Then
 \begin{itemize}
  \item[\rm (a)]  the number of simple factors in $M$ is at most $r$;
  \item[\rm (b)] the rank of every simple factor of $M$ is $(r, l)$-bounded, where $l$ is the maximum  of the  ranks of the factors  isomorphic to $PSL_2(q)$  with $q \equiv 7\,(\operatorname{mod} 8)$.
       \end{itemize}
\end{lemma}

\begin{proof}
Let $M=S_1\times \dots \times S_n$, where each $S_i$ is a simple normal subgroup. Each $S_iC$ contains $C$ as a Carter subgroup. The quotient $S_iC/C_{S_iC}(S_i)$ is an almost simple group. The image $\bar C $ of $C$ is a Carter subgroup of this quotient; note that the image of $S_i$ is isomorphic to $S_i$ and we denote it by the same symbol.

(a)  We claim that every $S_i$ contains an involution from $I_G(C)$.  According to \cite{vdo4} either $\bar C \cap S_i$ has even order, or
$S_i=PSL_2(3^{2k+1})$ and $\bar C$ contains a Sylow $3$-subgroup
of $C_{PSL_2(3^{2k+1})}(\zeta_{3'})$, where $\zeta_{3'}$ is the $3'$-part of a field automorphism $\zeta$ of odd order.
If $\bar C \cap S_i$ has even order, choose an involution $x\in \bar C \cap S_i$. By Glauberman's $Z^*$-theorem,
there is a  conjugate $x^g \ne x$ for $g\in S_i$ such that $x^gx$ is an involution, and it belongs to $S_i\cap I_G(C)$ because $x^gx =[g,x]=[g,{\hat x}]$ for a pre-image $\hat x\in C$ of $x$. When $S_i=PSL_2(3^{2k+1})$, there is an element $y$ of order 3 in $\bar C\cap  PSL_2(3)$, so that for some $g\in S_i$ the commutator $[g,y]=[g,\hat y]$ is an involution  in $S_i\cap I_G(C)$, where $\hat y\in C$ is a pre-image of $y$. Picking an involution in $S_i\cap I_G(C)$ for every $i$, we generate by them an elementary abelian 2-group of rank $n$, so that $n\leq r$ by hypothesis.

(b) By Corollary~\ref{c-simple} applied to $S_iC/C_{S_iC}(S_i)$, either the rank of $S_i$ is $r$-bounded, or $S_i = PSL_2(q)$ with $q \equiv 7\,(\operatorname{mod}8)$. Hence the result.
\end{proof}

In the interests of Corollaries~C and D we also prove an analogue of Lemma~\ref{112} with additional conditions which make it possible to drop the dependence on the ranks of composition factors isomorphic to $PSL_2(q)$ with $q \equiv 7\,(\operatorname{mod}8)$.

\begin{lemma}\label{112c} Let $G$ be a group containing a Carter subgroup $C$ such that any subgroup generated by a subset of $I_G(C)$ is $r$-generator. Suppose that $G=MC$, where $M$ is a direct product of simple groups, each normal in $G$.  Suppose in addition that one of the following holds:
\begin{itemize}
  \item[(1)]  $|C|$ is not divisible by 8; or
  \item[(2)]   for any $x \in I_G(C)$, any subgroup generated by a subset of $I_G(x)$ can be generated by $r$ elements.
\end{itemize}
 Then  the rank of $M$ is $r$-bounded.
\end{lemma}

\begin{proof}
  By Lemma~\ref{112} the number of simple factors of $M$ is at most $r$ and each factor either has $r$-bounded rank or is isomorphic to $PSL_2(q)$ with $q \equiv 7\,(\operatorname{mod}8)$. It remains to show that the latter factors also have $r$-bounded rank. Let $S=PSL_2(q)$ with $q \equiv 7\,(\operatorname{mod}8)$ be one of such factors of $M$. Then $SC$ contains $C$ as a Carter subgroup. The quotient $SC/C_{SC}(S)$ is an almost simple group, the image $\bar C $ of $C$ is a Carter subgroup of this quotient, and the image of $S$ is isomorphic to $S$ and is  denoted by the same symbol.

  Under condition (1), the group $G$ cannot have composition factors isomorphic to $S=PSL_2(q)$ with $q \equiv 7\,(\operatorname{mod}8)$, since the order $q(q-1)(q+1)$ of such a factor is divisible by~8.

  Now assume condition (2). According to \cite{vdo4}, either $S=PSL_2(3^{2k+1})$
  or $\bar C$ contains a Sylow 2-subgroup of $S\bar C$. The rank of $S=PSL_2(3^{2k+1})$ is $r$-bounded by Corollary~\ref{c-simple}.
   When $\bar C$ contains a Sylow 2-subgroup $T$ of $S\bar C$, the set $I_{S\bar C}(T)$ is a subset of  $I_{S\bar C}(\bar C)$ and therefore satisfies the hypotheses of Theorem~\ref{t-1-2}. By this theorem the rank of $[S,T]$ is $r$-bounded, and so is the rank of $S$.
\end{proof}

\section{Carter subgroups}\label{s-general}

Recall that the generalized Fitting subgroup $F^*(G)$ of a group $G$ is  equal to $F(G)E(G)$, where $F(G)$ is the Fitting subgroup, while $E(G)=Q_1*\dots *Q_k$, known as the layer, is the central product of all subnormal quasisimple subgroups $Q_i$, that is, perfect groups with non-abelian simple central quotients.

We denote by $R(G)$ the soluble radical of a group $G$, its largest normal soluble subgroup. First we prove Theorem~B in the case where the soluble radical $R(G)$ is trivial, obtaining in addition that then the generalized Fitting subgroup  has bounded index.

\begin{proposition}\label{113}
Let $G$ be a group containing a Carter subgroup $C$ such that any subgroup generated by a subset of $I_G(C)$ is $r$-generator. Suppose that $R(G)=1$. Then
\begin{itemize}
  \item[\rm (a)] the group $G$ has $(r, l)$-bounded rank, where $l$ is the maximum of the  ranks of the composition factors of $G$ isomorphic to $PSL_2(q)$  with $q \equiv 7\,(\operatorname{mod} 8)$.
  \item[\rm (b)] the generalized Fitting subgroup $F^*(G)$ has $(r,l)$-bounded index.
\end{itemize}
 \end{proposition}

\begin{proof}  Since  the soluble radical $R(G)$ of $G$ is trivial, the socle  $F^*(G)$  is a direct product of non-abelian simple groups, say, $ F^*(G)=T_1\times \dots \times T_n$, where  the $T_i$ are non-abelian simple groups. Recall that $G$ acting by conjugation permutes the factors $T_i$. It is convenient to write  $ F^*(G)=M\times N$, where $M$ is the product of the factors $T_i$ that are normalized by the Carter subgroup $C$, and $N$  is the product of the other factors $T_i$, which are not normalized by $C$ (it is possible that one of the subgroups $M,N$ is trivial).

\begin{lemma}\label{l-113-1}
 The number  $n$ of the simple factors  $T_i$ of $F^*(G)$ is
 at most $3r+1$.
\end{lemma}

\begin{proof}
Applying  Lemma~\ref{112}(a) to the  subgroup $MC$ we obtain that the number of simple factors of $M$ is at most $r$. It remains to bound the number of simple factors in $N$. Let $N=S_1\times\dots\times S_t$, where each of the simple factors $S_i$ is not normalized by $C$. We claim that $t\leq 2r+1$. Indeed, if $t\geq 2r+2$, then we can assume without loss of generality that there are factors $S_1,\dots ,S_{2r+2}$ and elements  $c_i\in C$ for $i=1,\dots,r+1$ such that $S_{2i-1}^{c_i}=S_{2i}$. In every ${S_{2i-1}}$ we choose an involution $x_i$ and set $y_i=[x_i,c_i]=x_i^{-1}x_i^{c_i}$. Since $x_i^{-1}\in S_{2i-1}$ and $x_i^{c_i}\in S_{2i}$ for all $i$, the subgroup $\langle y_1,\dots , y_{r+1}\rangle$ is an abelian 2-group of rank exactly $r+1$. Since this subgroup is generated by a subset of $I_G(C)$, we obtain a contradiction with the hypothesis. Thus the number  $n$ of simple factors of $F^*(G)$ is
at most $3r+1$.
\end{proof}

\begin{lemma}\label{l-113-2}
The  rank of each of the simple factors  $T_i$ of $F^*(G)$ is $(r,l)$-bounded.
\end{lemma}

\begin{proof}
The  rank of each simple factor of $M$ is $(r,l)$-bounded   by Lemma~\ref{112}(b) applied to $MC$. It remains to show that the  rank of each simple factor $S$ of $N$ is $(r,l)$-bounded; in fact, we claim that the rank of $S$ is  $r$-bounded.  By Lemma~\ref{l-rank-syl} the rank of a finite group is at most $m+1$, where $m$ is the maximum of the ranks of its Sylow subgroups.  Hence it suffices to prove that the rank of every Sylow $p$-subgroup $P$ of $S$ is  $r$-bounded, for every prime $p$. In turn, the rank of $P$ is bounded in terms of the rank of a maximal abelian normal subgroup $A$ of $P$ by Lemma~\ref{l-gmh}.  Let $c\in C$ be an element such that $S^c\neq S$. Then $[A,c]$ is an abelian  $p$-group generated by a subset of $I_G(C)$ and therefore has rank at most $r$. Since the projection of $[A,c]$ onto $S$ is $A$, the rank of $A$ is at most $r$. Hence the result.
\end{proof}

We return to the proof of Proposition~\ref{113}. It follows from Lemmas~\ref{l-113-1} and \ref{l-113-2}  that $F^*(G)$ has $(r, l)$-bounded rank. We now prove that the index of $F^*(G)$ is $(r, l)$-bounded, which will also imply that the rank of $G$ is $(r, l)$-bounded.

Since $R(G)=1$, the group $G$ embeds into the automorphism group of $F^*(G)$.  It is well known that $ \operatorname{Aut} F^*(G) = (\operatorname{Aut} T_1\times \dots \times \operatorname{Aut}T_n)U$,
 where $U$ is a subgroup of the symmetric group $S_n$. Since $n$ is $r$-bounded by Lemma~\ref{l-113-1}, it remains to prove that the orders of the outer automorphism groups $ \operatorname{Aut}T_i/T_i$ are $(r,l)$-bounded. This follows from Lemma~\ref{l-113-3}(c), since the ranks  of the $T_i$ are $(r, l)$-bounded by Lemma~\ref{l-113-2}.

Thus, $|G:F^*(G)|$ is $(r, l)$-bounded, and therefore the rank of $G$ is $(r, l)$-bounded.
\end{proof}

Under the same hypotheses as in Proposition~\ref{113}, the following lemma provides additional information for the case where the soluble radical is trivial: then $G$ can be generated by $(r,l)$-boundedly many  conjugates of $C$.

\begin{lemma}\label{114} Let $G$ be a group containing a Carter subgroup $C$ such that any subgroup generated by a subset of $I_G(C)$ is $r$-generator. Suppose that $R(G)=1$. Then there are $(r, l)$-boundedly many conjugates of $C$ which generate the group $G$,  where $l$ is the maximum of the  ranks of the composition factors of $G$ isomorphic to $PSL_2(q)$  with $q \equiv 7\,(\operatorname{mod} 8)$.
\end{lemma}

\begin{proof}  Since $R(G)=1$, the  generalized Fitting subgroup $F^*(G)=S_1\times\dots\times S_n$   is a direct product of non-abelian simple groups $S_i$.  The number $n$ of these simple factors is at most $3r+1$  by  Lemma~\ref{l-113-1}, and each factor has $(r, l)$-bounded rank by Proposition~\ref{113}(a). Set $C_0=C\cap F^*(G)$.  We claim that   $F^*(G)$ is generated by $(r,l)$-boundedly many conjugates of~$C_0$.

Let $F^*(G)=M_1\times\dots\times M_k$, where each $M_j=S_{j1}\times\dots\times S_{jk_j}$ is a direct product of the simple factors of $F^*(G)$ that form one orbit under the action of $C$ on the set of simple factors of $F^*(G)$ induced  by conjugation. Let $C_j=C\cap M_j$ for $j=1,\dots,k$. Observe that  $C_j\neq1$ for each $j$. Indeed, $C$ is a nilpotent self-normalizing group acting on $M_j$. If $C\cap M_j=1$, then  $C_{M_j}(C)=1$, which would imply that $M_j$ is soluble by Theorem~\ref{t-beha}, a contradiction.

The subgroup $C$ transitively permutes the factors $S_{j1},\dots,S_{jk_j}$ of $M_j$ and so the projection of $C_j$ onto each factor is nontrivial.
 Since each  factor $S_{js}$ has $(r,l)$-bounded rank, by Theorem~\ref{t-hls}  the factor $S_{js}$ can be generated by $(r,l)$-boundedly many conjugates of the projection of $C_j$ onto $S_{js}$, say, conjugates by elements $a_1,\dots,a_t
 \in S_{js}$, where $t$ is $(r,l)$-bounded. (Of course, these elements $a_i$ are different for different~$j$.) The subgroup $H_j=\langle C_j^{a_1},\dots , C_j^{a_t}\rangle$ has the same projections as $C_j$ onto all the other simple factors in $M_j$. Since $C_j$ is nilpotent, it follows that some term $H_j^{(m)}$ of the derived series of $H_j$ has trivial projections onto all these other  factors of $M_j$, while the projection of $H_j^{(m)}$  onto   $S_{js}$ is $S_{js}$. Therefore $H_j^{(m)}=S_{js}$; in particular, $S_{js}\leq H_j$.

Since the total number of simple factors of  $F^*(G)$ is $r$-bounded, we obtain that  $F^*(G)$ can be generated by $(r,l)$-boundedly many conjugates of $C_0$.

Since $|G:F^*(G)|$ is $(r,l)$-bounded by Proposition~\ref{113}(b) and $G$ is the normal closure of $C$ by Corollary~\ref{c-cart}, the result follows.
\end{proof}

We can now complete the  proof of Theorem~B, which we restate here for convenience.

\begin{theorem}\label{t-main}
Let $r$ and $l$ be positive integers. Suppose that  a finite group $G$ contains a Carter subgroup $C$ such that any subgroup generated by a subset of $I_G(C)$ can be generated by $r$ elements. Let $l$ be the maximum
rank of composition factors of $G$ isomorphic to $PSL_2(q)$ for $q\equiv 7\,(\operatorname{mod}8)$. Then the derived subgroup  $G'$ has $(r,l)$-bounded rank.
\end{theorem}

\begin{proof}
Recall that $G$ is  a  group containing a Carter subgroup $C$ such that any subgroup generated by a subset of $I_G(C)$ is $r$-generator. We want to show that  $G'$ has $(r,l)$-bounded rank, where $l$ is the maximum of the  ranks of the composition factors of $G$ isomorphic to $PSL_2(q)$  with $q \equiv 7\,(\operatorname{mod} 8)$.

Let $R=R(G)$ be the soluble radical of $G$. By Lemma~\ref{114} there are $(r,l)$-boundedly many conjugates $C_1,\dots,C_m$ of $C$ such that $G=R\langle C_1,\dots,C_m\rangle $.  Observe that $C_i$ is a Carter subgroup of $RC_i$ for any $i=1,\dots,m$. Hence  $[RC_i,C_i]$ has $r$-bounded rank by Theorem~\ref{t-sol}. In particular,  each $[R,C_i]$ has $r$-bounded rank. Since the $[R,C_i]$ are normal subgroups of $R$, it follows that the product
$$
L=\prod_{i=1}^{m}{[R,C_i]}
$$
has $(r,l)$-bounded rank, since $m$ is  $(r,l)$-bounded.

We claim that $L$ is a normal subgroup of $G$.
Indeed,  $L$ is normal in $R$, and $L$ is $C_j$-invariant for every $j$ since it contains $[R,C_j]$; hence $L$ is normal in $G=R\langle C_1,\dots,C_m\rangle $.

Since $L$ has $(r,l)$-bounded rank, we can pass  to the quotient $G/L$ and thus  assume that $[R,C_j]=1$ for every $i$. (We used the fact that the image of $R$ in $G/L$ is the soluble radical of $G/L$.) Since $R$ is normal, we obtain that $[R,C_j^g]=1$ for any conjugates $C_j^g$. Since the group $G$ is generated by the conjugates of a Carter subgroup by Corollary~\ref{c-cart}, it follows that then $R$ is central in $G$.

By Proposition~\ref{113} the quotient $G/R$ has $(r,l)$-bounded rank. Since $R$ is central in $G$, the rank of $G'$ is $(r,l)$-bounded by  a theorem of Lubotzky and Mann \cite[Theorem~4.2.3]{LM}.
\end{proof}

\begin{proof}[Proofs of Corollaries~C and D]
The proof of these corollaries is obtained by repeating word-for-word the proof of Theorem~\ref{t-main}, with ``$(r,l)$-bounded" replaced everywhere with ``$r$-bounded", and with the references to Lemma~\ref{112} replaced with the references to Lemma~\ref{112c}.
\end{proof}

\section{Counterexamples}\label{counterexamples}

A minor variation of the arguments in \cite[\S\,5]{agks} gives the following.

\begin{lemma} \label{l-6.1}
Let $G= PSL_2(p^e)$, where $p$ is a prime such that  $p \equiv 7\,(\operatorname{mod}16)$ and $e$ is odd.   If $x \in G$ is a $2$-element, then every subgroup
generated by a subset of $I_G(x)$ can be generated by $2$ elements.
\end{lemma}

\begin{proof}
With a slight abuse of notation we will work in $G=SL_2(p^e)$ denoting by $x$ a $2$-element that is a pre-image of a given $2$-element of $PSL_2(p^e)$. Let $H$ be a subgroup generated by a subset of $I_G(x)$; we need to show that $H$ is 2-generated. Note that any subgroup of $G$ not contained in a Borel subgroup is generated by at most two elements. Indeed, the subgroups of~$SL_2(p^e)$ are either contained in a Borel subgroup, or are dihedral, or are contained in a subgroup isomorphic to  $S_4$ or $A_5$, or are isomorphic to $PSL_2(p^f)$ for some $f$  dividing $e$. So it remains to consider the case where $H \le B = UT$, where  $B$ is a Borel subgroup,  $U$ its unipotent radical, and $T$ its maximal torus. Note that the conditions $p \equiv 7\,(\operatorname{mod}16)$ and $e$ odd mean that the image of $B$ in $PSL_2(p^e)$ has odd order and the Sylow $2$-subgroup of $SL_2(p^e)$  is generalized quaternion of order $16$ defined over the prime field (as a subgroup of $SL_2(p)$). To lighten the notation, we sometimes write $q=p^e$; note that $q \equiv 7\,(\operatorname{mod}  16)$.

First we consider the case where the image of $x$ is an involution in $PSL_2(q)$.      We observe that   $U \cap I_G(x) =\{1\}$, since $x$ cannot invert a nontrivial
element of~$U$ because $B$ has odd order over the centre.  We have $I_G(x) \cap B = T_x$, where $T_x$ is the (unique) maximal torus of $B$ that is inverted by $x$; actually,  $T_x = B \cap B^x$. Thus any subgroup generated by a subset of $B \cap I_G(x)$ is cyclic, and the case of an involution in $PSL_2(q)$ is complete.

From now on we assume that $x^2$ is not central in $G=SL_2(q)$. Note that then $x$ has order 8 and is conjugate to $x^{-1}$ in a generalized quaternion Sylow $2$-subgroup. The centralizer $C_G(x)$ is a maximal nonsplit torus of order $q+1$.  The action of $GL_2(q)$ by conjugation  induces an action of $PGL_2(q)$ on $SL_2(q)$. The centralizer of $x$ in $PGL_2(q)$ is also a torus of order $q+1$.

Recall that it suffices to show that any subgroup generated by a subset of $I_G(x) \cap B$ is generated by two  elements.

Let $t $ be any noncentral semisimple element
of $B$.  Consider the set of triples $(a,b,c)$ of elements of $G$ with  $abc=1$ such that $a$ and $b$ are conjugate to $x$  and $c$ is conjugate to $t$.  This set is nonempty by \cite{Gow, GT}.  Since the centralizer of $x$ in $PGL_2(q)$ is a torus of order $q+1$ and the centralizer of $t$ is a torus of order $q-1$, the centralizer of the subgroup generated by any conjugate of $x$ and  any conjugate of $t$ is trivial in $PGL_2(q)$. Clearly, any such subgroup is irreducible and so absolutely irreducible. It follows by \cite[Theorem~2.3]{SV} that all such triples form a regular $PGL_2(q)$-orbit.

Recall that $x$ is conjugate to $x^{-1}$.
If we only consider such triples with $a=x^{-1}$, they form a single orbit under $C$, the centralizer of $x$ in $PGL_2(q)$. So there are $q+1$  triples   $(x^{-1},b,c)$ with $x^{-1}bc=1$ such that $b$ is conjugate to $x$ and $c$ is conjugate to $t$.
Note that for such triples we have  $c=t^u=(x^{-1}x^v)^{-1}=[v,x]\in I_G(x)$ for some $u,v\in G$. Since $C$ acts regularly on
the set of Borel subgroups, we obtain  $|B \cap I_G(x) \cap t^G| =1$.  Thus there is a unique conjugate $s$ of $t$ such that $s=[v,x]\in I_G(x)$.

If we replace $t$ with a nontrivial unipotent element, in fact there are no solutions~\cite{gu1}. The only other noncentral class is the class of  $-u$ for $u$ a nontrivial unipotent element. The argument above applies to that class as well and so $|I_G(x) \cap B \cap(-u)^G|\leq 1$ and $I_G(x) \cap B \cap (-u)^G \subset SL_2(p)$.

We now claim that if $y \in I_G(x)\cap B$ is conjugate to an element $z$ of
$SL_2(p^f)$ for some $f$ dividing $e$, then $y \in SL_2(p^f) \cap B$. Indeed, we can assume that $z$ is a semisimple element in a Borel subgroup of $ SL_2(p^f)$. Since $x\in SL_2(p)$, by the above argument applied to $SL_2(p^f)$ there is a conjugate of $z$ within $SL_2(p^f)$ that belongs to $I_G(x)$. By the uniqueness established in the preceding paragraph, this conjugate must coincide with~$y$, so $y\in SL_2(p^f)$.

Consider some subset $S$ of $I_G(x) \cap B$; we need to prove that $\langle S \rangle$ can be generated by two elements. Write the nontrivial elements of $S$ as $w^{j_k}u$ with $1 \le j_k < q-1$ and with various $u \in U$, for which we do not use indices to lighten the notation.

Suppose that for an element $w^{j_k}u\in S$ we have $w^{j_k} \in SL_2(p^f)$ for some $f$ dividing $e$. In a suitable basis, $w^{j_k}$ is a diagonal matrix, while $u$ is upper unitriangular. There is a diagonal matrix $d$ such that the conjugate $d^{-1}ud$ belongs to $SL_2(p)$. Namely, since $q \equiv 7\,(\operatorname{mod}  16)$,  we have  $u=I_2+\pm\alpha e_{12}$, where $\alpha$ has odd order in $\mathbb{F}_q^*$ and so there is $\beta\in \mathbb{F}_q$ such that $\beta ^2=\alpha$. Then for $d=\operatorname{diag} (\beta, \beta ^{-1})$ we have $d^{-1}ud=I_2+\pm e_{12}\in SL_2(p)$; note also that $d$ commutes with $w^{j_k}$. As a result, $d^{-1}w^{j_k}ud$ belongs to $SL_2(p^f)$, whence  $w^{j_k}u$ belongs to $SL_2(p^f)$ by the above.

Let $f$ be the smallest positive integer such that  $SL_2(p^f)$ contains all the factors $w^j$ of the elements $w^ju$ of $S$. Then all the elements $w^ju$ of $S$ also belong to $SL_2(p^f)$ as shown above. It follows that $\langle S \rangle/(\langle S \rangle\cap U)$ acts irreducibly on $\langle S \rangle\cap U$ regarded as a vector space over the prime field $\mathbb{F}_p$, so that $\langle S \rangle\cap U$ is a cyclic $\mathbb{F}_p(\langle S \rangle/(\langle S \rangle\cap U))$-module. Since $\langle S \rangle/(\langle S \rangle\cap U)$ is cyclic, it follows that $\langle S \rangle$ is $2$-generator.

Thus, we have shown that any subgroup of $G$ generated by a subset of $I_G(x)$
can be generated by at most two elements.
\end{proof}

\begin{example}  Let $G = PSL_2(p^e)$ with $p \equiv 7\,(\operatorname{mod}  16)$ and $e$ odd.   Then a Sylow $2$-subgroup $K$ of $G$ is dihedral
of order $8$ and is a Carter subgroup of $G$ by \cite{vdo4}. We claim that every subgroup of $G$ generated by a subset of $I_G(K)$ can be generated by
$14$ elements. At the same time, the rank of $G$ is unbounded, as it is at least $e$. Thus, the bound for the rank in Theorem~B cannot depend on $r$ alone.

To prove the above claim, we note that $S=\bigcup_{x\in K} S\cap I_G(x)$. Each subgroup $\langle S\cap I_G(x)\rangle$ can be generated by 2 elements by Lemma~\ref{l-6.1}. Since $K$ has seven nontrivial elements, the subgroup $\langle S\rangle$ can be generated by $14$ elements.
\end{example}

\newpage

\section*{Supplementary table to \cite{vdo4}}
\label{s-sup}

This is the amended Table~8 of \cite{vdo4} (its title ``Table~10'' corresponds to the ArXiv version of  \cite{vdo4}). The first column contains a simple group $S$ such that the Carter subgroups of $\mathop{Aut}(S)$ are classified. The group of inner-diagonal automorphisms of $S$ is denoted by $\widehat{S}$. The second column contains the conditions for
a subgroup $A\leqslant   \mathop{Aut}(S)$ to contain a Carter subgroup. The third column describes the structure of a
Carter subgroup~$K$. The subgroups of $\mathop{Aut}(S)$ strictly between $S$ and $A$ have no Carter subgroups. A Sylow $r$-subgroup of a group $G$ is denoted by $P_r(G)$. A field automorphism of a group
of Lie type $S$ is denoted by $\varphi$, and  a graph automorphism of $S$
contained in $K$ is denoted  by $\tau$ (since the graph
automorphisms of order $2$ and $3$ of $D_4(q)$ do not commute, only one of them can be in $K$). If $A$ does not
contain a graph automorphism, then we assume that $\tau  = 1$. A field automorphism of $S$ of
maximal order contained in $A$ is denoted by $\psi$ (it is a power of $\varphi$, but $\langle\psi \rangle$ may different from $\langle\varphi\rangle$). A Carter
subgroup of order $2\cdot 3$ of $\widehat{{}^2A_2(2)}$ is denoted by $K(U_3(2))$. If $G$ is solvable,  then $K(G)$ denotes a Carter subgroup of $G$.   A~graph-field automorphism of order $2t$ of $A_2(2^{2t})$ is denoted by $\zeta$.

 \hspace{-4.5cm} \includegraphics[width=1.3\textwidth]{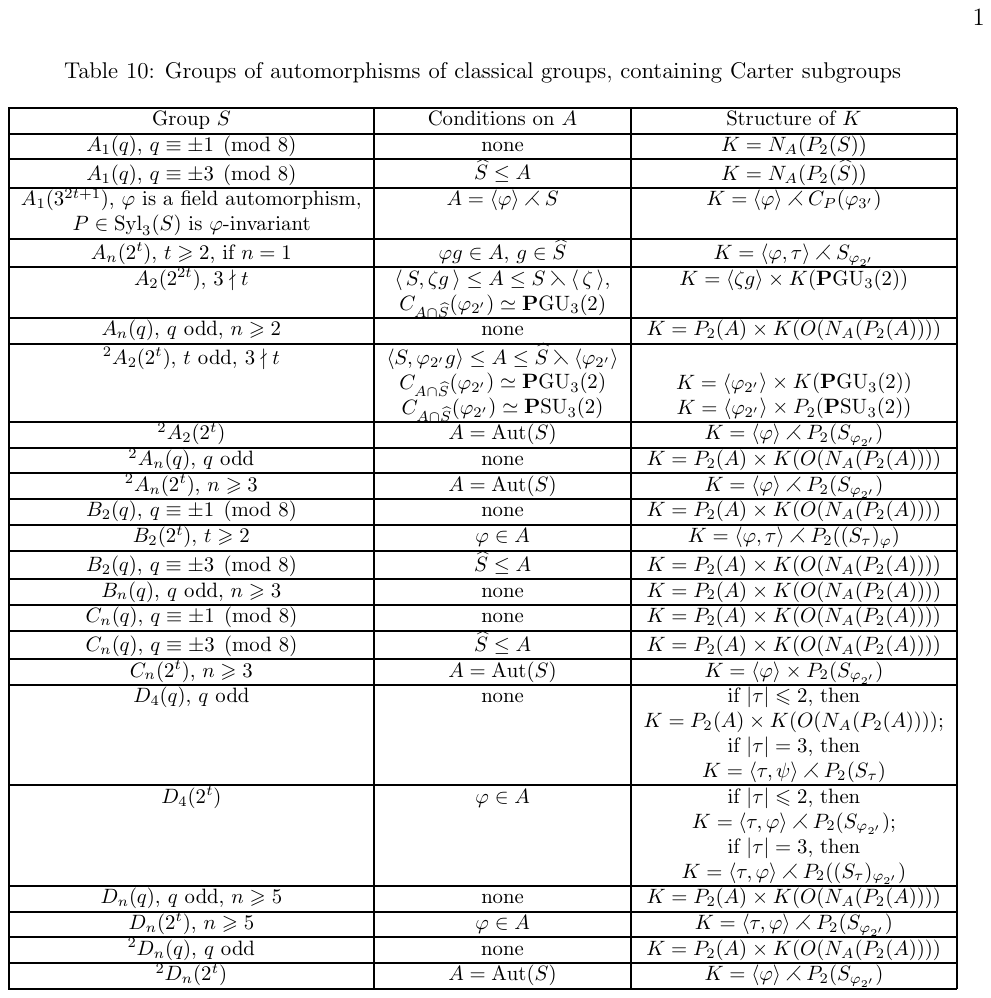}

\end{document}